\documentclass[a4paper,oneside,11pt]{article}

\usepackage{amsmath,amssymb,amsthm}
\usepackage[hidelinks]{hyperref}
\usepackage{yfonts}
\usepackage{enumerate}
\usepackage{color}
\usepackage{mathtools}
\usepackage{thmtools,thm-restate}
\usepackage{tikz}
\usepackage[justification=centering]{caption}
\usepackage{ wasysym } 
\usepackage{verbatim}
\usepackage{bbm}

\usepackage[ruled]{algorithm2e}
\SetArgSty{textup}

\usepackage{geometry}

\newcommand{\OO}{\mathcal{O}}

\newcommand{\myimg}[2]{\IfFileExists{#1}{\includegraphics[scale=#2]{#1}   }{  Image not found  \errmessage{image #1 is missing}}}
\renewcommand{\epsilon}{\varepsilon}

\makeatletter
\newcommand*{\da@rightarrow}{\mathchar"0\hexnumber@\symAMSa 4B }
\newcommand*{\da@leftarrow}{\mathchar"0\hexnumber@\symAMSa 4C }
\newcommand*{\xdashrightarrow}[2][]{%
  \mathrel{%
    \mathpalette{\da@xarrow{#1}{#2}{}\da@rightarrow{\,}{}}{}%
  }%
}
\newcommand{\xdashleftarrow}[2][]{%
  \mathrel{%
    \mathpalette{\da@xarrow{#1}{#2}\da@leftarrow{}{}{\,}}{}%
  }%
}
\newcommand*{\da@xarrow}[7]{%
  \sbox0{$\ifx#7\scriptstyle\scriptscriptstyle\else\scriptstyle\fi#5#1#6\m@th$}%
  \sbox2{$\ifx#7\scriptstyle\scriptscriptstyle\else\scriptstyle\fi#5#2#6\m@th$}%
  \sbox4{$#7\dabar@\m@th$}%
  \dimen@=\wd0 %
  \ifdim\wd2 >\dimen@
    \dimen@=\wd2 %
  \fi
  \count@=2 %
  \def\da@bars{\dabar@\dabar@}%
  \@whiledim\count@\wd4<\dimen@\do{%
    \advance\count@\@ne
    \expandafter\def\expandafter\da@bars\expandafter{%
      \da@bars
      \dabar@ 
    }%
  }%
  \mathrel{#3}%
  \mathrel{%
    \mathop{\da@bars}\limits
    \ifx\\#1\\%
    \else
      _{\copy0}%
    \fi
    \ifx\\#2\\%
    \else
      ^{\copy2}%
    \fi
  }%
  \mathrel{#4}%
}

\makeatother

\newcommand{\srightarrow}[1]{\xdashrightarrow{#1}}
\newcommand{\brightarrow}[1]{\overset{#1}{\rightsquigarrow}}

\newcommand{\A}{\mathbf{A}}
\newcommand{\B}{\mathbf{B}}
\newcommand{\C}{\mathbf{C}}
\newcommand{\D}{\mathbf{D}}
\newcommand{\E}{\mathbf{E}}

\newcommand{\PP}{\mathbf{P}}
\newcommand{\MM}{\mathbf{M}}

\newtheorem{theorem}{Theorem}[section]
\newtheorem{lemma}[theorem]{Lemma}
\newtheorem{proposition}[theorem]{Proposition}

\newtheorem{claim}[theorem]{Claim}

\newtheorem{remark}[theorem]{Remark}

\newtheorem{observation}{Observation}

\theoremstyle{definition}
\newtheorem{definition}[theorem]{Definition}
\newtheorem{example}[theorem]{Example}

\title{Optimal schemes for combinatorial query problems with integer feedback
}

\author{ 
Anders Martinsson%
\thanks{Department of Computer Science, ETH Zurich, Switzerland.
        Email: anders.martinsson@inf.ethz.ch} 
}

\date{\today}

\begin{document}
  \maketitle

\begin{abstract}
A \emph{query game} is a pair of a set $Q$ of \emph{queries} and a set $\mathcal{F}$ of functions, or \emph{codewords} $f:Q\rightarrow \mathbb{Z}.$ We think of this as a two-player game. One player, Codemaker, picks a hidden codeword $f\in \mathcal{F}$. The other player, Codebreaker, then tries to determine $f$ by asking a sequence of queries $q\in Q$, after each of which Codemaker must respond with the value $f(q)$. The goal of Codebreaker is to uniquely determine $f$ using as few queries as possible. Two classical examples of such games are coin-weighing with a spring scale, and Mastermind, which are of interest both as recreational games and for their connection to information theory.

In this paper, we will present a general framework for finding short solutions to query games. As applications, we give new self-contained proofs of the query complexity of variations of the coin-weighing problems, and prove new results that the deterministic query complexity of Mastermind with $n$ positions and $k$ colors is $\Theta(n \log k/ \log n + k)$ if only black-peg information is provided, and $\Theta(n \log k / \log n + k/n)$ if both black- and white-peg information is provided. In the deterministic setting, these are the first up to constant factor optimal solutions to Mastermind known for any $k\geq n^{1-o(1)}$.
\end{abstract}

\section{Introduction}

In 1960, Shapiro and Fine \cite{SFproblemE1399} posed the following question. Given $n$ coins of unknown composition, where counterfeit coins weigh $9$ grams and genuine coins weigh $10$ grams, and an accurate spring scale (not a balance), how many weighings are required to isolate the counterfeit coins? For $n\geq 4$, Shapiro and Fine observed that this can be done in fewer than $n$ weighings, leading to the conjecture that, in general, $o(n)$ weighings would suffice.

The problem of determining the optimal number of weighings became the subject of much independent research in the subsequent years. To summarize the findings, we distinguish between \emph{static} solutions where the weighings are fixed before learning the answers, and \emph{adaptive} solutions where each weighing may depend on answers to previous weighings. The minimum number of weighings required in static solutions is known to be
$$ \left(2 + \OO\left(\frac{\log \log n}{\log n}\right)\right) \frac{n}{\log_2 n},$$
where the lower bound is due to \cite{ERcoin,MoserL} and the upper bound is due to constructions given by Cantor and Mills \cite{cantor1966determination} and Lindstr\"om \cite{lindstrom1964combinatory,lindstrom1965combinatorial}. It was also shown by Erd\H{o}s and R\'enyi \cite{ERcoin} that a sequence of weighings chosen uniformly at random is likely to be able to distinguish between all configurations of counterfeit coins, albeit at a slightly higher number of $(\log_2 9 + o(1))n/\log_2 n$ weighings. See also the last section of their paper for an overview of further early works on this topic.

While one would expect adaptive solutions to be more powerful, they are also harder to analyze. In either the static or adaptive case, the best known lower bounds are established by considering upper bounds on the amount of information that can be obtained in a single weighing, assuming the coins are independently either counterfeit or genuine, each with probability $\frac12$. In the static case, the results of each weighing follow a binomial distribution, which is concentrated on $\sqrt{n}$ values, limiting the query entropy to roughly $\log_2 \sqrt{n} = \frac12 \log_2 n$. However, in the adaptive case, the distribution of values returned from the weighings depends on how queries are generated. It is reasonable to think that optimal adaptive solutions would make use of this to make queries with higher information content, but to what extent this can be done remains unknown. The best known bounds on the minimum number $T$ of queries needed in the adaptive case is
$$\frac{n}{\log_2 (n+1)} \leq T \leq \left(2+\OO\left(\frac{\log\log n}{\log n}\right)\right)\frac{n}{\log_2n},$$
where the upper bound is given by the static solutions, and the lower bound is obtained by bounding the query entropy by $\log_2(n+1)$, which is an upper bound on the entropy of any distribution on $\{0, 1,\dots, n\}$. This leaves a gap of a factor of $2$ between the upper and lower bounds, which remains an important open question to this day.

To better understand the difference between static and adaptive solutions, it is useful to examine games where this contrast is more apparent. An example of such a game is the popular board game Mastermind, which was introduced commercially in the 1970s and formally described in 1983 by Chv\'atal \cite{chvatal1983mastermind}. In Mastermind, one player, known as the Codemaker, creates a hidden pattern, or codeword, consisting of $n$ entries, each of which can be one of $k$ different colors. The other player, known as the Codebreaker, must try to guess this pattern by making a series of guesses. After each guess, the Codebreaker is awarded a certain number of black and white pegs, which indicate how close the guess is to the hidden pattern. Denoting the codeword $x$ and a query $q$ as elements of $[k]^n$, the corresponding number of black pegs
$$bp(x, q) := |\{i\in[n] : x_i=q_i\}|$$
denotes the number of entries in the guess that match the corresponding entries in the hidden pattern\footnote{Later in the paper, when treating general query games, we will identify codewords with functions mapping queries to the integers. For instance, in Mastermind with black-peg feedback, we would identify the codeword $x$ with the map $bp(x, \cdot)$.}, while the number of white pegs 
$$wp(x, q) := \max_{\sigma} |\{i\in[n] : x_i=q_{\sigma(i)}\}| - bp(x, q),$$
where $\sigma$ goes over all permutations of $[n]$, denotes the the number of entries that have the correct color but are in the wrong position. The game ends when the Codebreaker correctly guesses the hidden pattern.

We will below distinguish between \emph{black-peg Mastermind}, where the Codebreaker is only told the number of black pegs corresponding to a query, and \emph{white-peg Mastermind}, where the Codebreaker is also told the number of white pegs.

Since Mastermind with two colors and coin-weighing are essentially equivalent, see \cite{ST04}, Mastermind can be seen as a direct extension of coin-weighing. Analogously to the entropy bound for coin-weighing, we can derive lower bounds on the  minimum number of adaptive queries $T$ needed to recover any codeword. For black-peg Mastermind, this yields the bound
$$\frac{n\log_2 k}{\log_2(n+1)} \leq T,$$ 
while for white-peg Mastermind, this yields the bound
$$\frac{n \log_2 k}{\log_2 {n+2\choose 2}} \leq T.$$
We note that these bounds only differ by a factor of $2+o(1)$.
 
For $2\leq k \leq n^{1-\varepsilon}$ for any fixed $\varepsilon > 0$, solutions also behave similar to coin-weighing. Chv\'atal \cite{chvatal1983mastermind} showed that, for any $k$ in this range, there exists a sequence of $$T= (2+\varepsilon)n \frac{1+2\log_2 k}{\log_2 n - \log_2 k}$$ static queries that uniquely determines any codeword, using only black-peg information. The proof uses the probabilistic method in a similar way to the proof by Erd\H{o}s and R\'enyi \cite{ERcoin} for coin-weighing. As this matches the general lower bounds up to a constant factor, this means that, regardless of whether the queries are static or adaptive, and whether black-peg or white-peg Mastermind is considered, the minimum number of queries needed to recover any codeword is $\Theta_\epsilon(n \log k/ \log n),$ with implicit constants depending on $\varepsilon>0.$

Static black-peg Mastermind has been the focus of interest in its own right due to its equivalent formulation as the metric dimension or rigidity of Hamming graphs. By extending the coin-weighing solution of Lindstr\"om \cite{lindstrom1964combinatory,lindstrom1965combinatorial}, Jiang and Polyanskii \cite{smallstaticmm2} showed that for any constant $k\geq 2$, 
$$ \left(2 + \OO\left(\frac{\log \log n}{\log n}\right)\right)\frac{n \log_2 k}{\log_2 n}$$
static guesses are optimal. 

For larger $k$, progress towards determining the length of optimal adaptive solutions has been slower. While the information-theoretic lower bound is linear in $n$ whenever $k\sim n^\alpha$ for any constant $\alpha> 0$, Chv\'atal states that the best upper bound he can offer for $k\geq n$ is $\OO(n \log n + k/n)$ in the white-peg setting and $\OO(n \log n + k)$ in the black-peg setting. Variations of these bounds with improved constants were proposed by Chen, Cunha and Homer \cite{CCH96}, Goodrich \cite{Goo09}, and J\"ager and Peczarski \cite{JP11}. It can be noted that the terms $k/n$ and $k$ in the above bounds are necessary, as, before guessing a color $x_i$ in the codeword right at least once, Codebreaker has no way of distinguishing between two colors it has yet to try. This means that the solutions above are all optimal up to constants if $k= \Omega(n^2\log n)$ for white-peg Mastermind, or $k=\Omega(n\log n)$ for black-peg Mastermind, but in particular for $k=n$ this still leaves a gap of a factor $\Theta(\log n)$.

In 2016, Doerr, Doerr, Sp\"ohel  and Thomas \cite{DDST16} first managed to narrow this gap by showing that black-peg Mastermind with $k=n$ colors can be solved in $\OO(n \log \log n)$ queries. Finally, a recent work of the author together with Su \cite{MSup22+} construct a randomized solutions for black-peg Mastermind with $k=n$ colors using $\OO(n)$ queries. This randomized solution has the property that, for any choice of the hidden codeword, Codebreaker will be able to uniquely identify it with probability at least $1-e^{-\Omega(n)}$. This directly implies the existence of a randomized solution to black-peg Mastermind for any $k\geq 2$ using $\OO( n + k )$ queries. Moreover, combining this with results from Doerr, Doerr, Sp\"ohel  and Thomas, it follows that there exists a randomized solution to white-peg Mastermind using $\OO(n + k/n)$ queries. Both of these results are best possible up to constant factors for $k\geq n^\varepsilon$ for any $\varepsilon>0$.

The difference between the cases where $k\leq n^{1-\varepsilon}$ and $k= n$ in Mastermind can be explained by the concept of adaptivity. In the former case, optimal static solutions are up to constant factors as short as adaptive ones. However, for larger values of $k$, this is no longer true. In fact, Doerr, Doerr, Sp\"ohel  and Thomas showed that $\Theta(n \log n)$ queries are optimal for static Mastermind with $k=n$ colors. The reason for this is that in the static case, assuming the codeword is chosen uniformly at random, the number of black pegs returned from any query follows a Bin$(n, 1/k)$ distribution, which has $\Theta_\varepsilon(\log n)$ entropy when $k\leq n^{1-\varepsilon}$, but only $\Theta(1)$ when $k=n$. Adding white-peg information does not change this picture unless $k$ is large. White pegs only carry information about the codeword up to permutations, which has negligible entropy compared to the full codeword when $k=\OO(n)$. Therefore, in order to match the information-theoretic lower bound for Mastermind for $k = n$ in either the black-peg or white-peg setting, one needs a precise understanding of the connection between adaptivity and information content of queries.

In this paper, we will take a more general approach to guessing games with integer feedback, such as coin-weighing and Mastermind. We will introduce a framework that allows us to express various guessing games, and adaptive solutions thereof in a common language. A virtue of the framework is that the information-theoretic speedup can be done using this abstract language, which makes it possible to separate the problem of making information-dense queries from problems relating to the specific game at hand. Our main technical results, Theorems \ref{thm:mainsimple} and \ref{thm:mainbounded}, capture this by giving sufficient conditions for games to have short adaptive solutions.

As applications of our main results, we present a sequence of self-contained proofs of the existence of, up to constants, optimal deterministic adaptive solutions for various guessing games. In particular, using our results, we are able to determine the deterministic query complexity of Mastermind in both the black-peg and white-peg settings.
\begin{theorem}\label{thm:blackpegmm} The minimum number of adaptive queries needed to solve black-peg Mastermind with $n$ positions and $k\geq 2$ colors is $\Theta\left( n\frac{\log k}{\log n} + k\right)$. 
\end{theorem}
\begin{theorem}\label{thm:whitepegmm} The minimum number of adaptive queries needed to solve white-peg Mastermind with $n$ positions and $k\geq 2$ colors is $\Theta\left(n\frac{\log k}{\log n} + k/n\right)$.
\end{theorem}
This matches the performance of the randomized solutions proposed by the author and Su \cite{MSup22+}. In the deterministic setting, these are the first known optimal solutions for any $k\geq n^{1-o(1)}$.

It can further be remarked that the proofs of Theorems \ref{thm:blackpegmm} and \ref{thm:whitepegmm} are constructive and the corresponding solutions can be turned into efficient algorithms, including efficient decoding of the hidden codeword. In this respect, our results also improve upon the previous solutions found through the probabilistic method by Chv\'atal in the range $2\leq k \leq n^{1-\varepsilon}$. It is worth noting that the algorithmic problem of determining the hidden codeword from a general sequence of query-answer pairs is known to be NP-hard in both the black-peg and white-peg settings \cite{SZ06,Goo09,Vig11}. Therefore, efficient recovery of the codeword is not immediate even in the cases where it is uniquely determined.

As part of deriving the solution for white-peg Mastermind, we further determine the query complexity of a new sparse version of the coin-weighing problem called \emph{sparse set query}, where the player is faced with a collection of $k$ coins and is told at most $n$ are counterfeit, for some parameters $k\geq n$. The problem is to determine which coins are counterfeit through few weighings, under the condition that at most $n$ coins may be placed on the scale at a time.
\begin{theorem}\label{thm:sparsesetquery}
For any $n \leq k$, the minimum number of adaptive queries needed to solve sparse set query is $\Theta\left(n \frac{\log(1+k/n)}{\log n} + k/n\right)$.
\end{theorem}

As a warm-up to proving these results, we will further derive up to constant factor optimal solutions to variations of the coin-weighing problem, and give a simplified proof for the existence of an $\OO(n)$ solution to black-peg Mastermind with $k=n$ in the case where the hidden codeword is required to be a permutation (but queries may still contain repeated colors).

In a sibling paper, we will further discuss how to use this framework to find optimal solutions for Permutation Mastermind, where $k=n$ and guesses and codewords are both forbidden to contain repeated colors.

The remaining part of the paper will be structured as follows. In Section \ref{sec:querygame} we will introduce the abstract framework and the two technical main results, Theorems \ref{thm:mainsimple} and \ref{thm:mainbounded}. In Section \ref{sec:applications} we will discuss applications of the main technical results, including Theorems \ref{thm:blackpegmm}, \ref{thm:whitepegmm} and \ref{thm:sparsesetquery}. Theorems \ref{thm:mainsimple} and \ref{thm:mainbounded} will be proven in Sections \ref{sec:simpleproof} and \ref{sec:boundedproof} respectively. Finally, in Section~\ref{sec:conclusion} we make some concluding remarks about our results.
\section{Query games}\label{sec:querygame}

A \emph{query game}, $A=(Q, \mathcal{F})$, is a pair of a set $Q$, and a non-empty set of functions $\mathcal{F}$ where $f:Q\rightarrow \mathbb{Z}$ for all $f\in\mathcal{F}$. We call the elements of $Q$ \emph{queries} and the elements of $\mathcal{F}$ \emph{codewords}. For technical reasons, we always require that query games contain a zero-query $0\in Q$ such that $f(0)=0$ for all $f\in \mathcal{F}$.

We can think of $A=(Q, \mathcal{F})$ as a two-player game as follows. The first player, Codemaker, first chooses a codeword $f\in \mathcal{F}$ without telling the other player the choice. The second player, Codebreaker, must then determine $f$ by asking a series of queries $q(1), q(2),\dots $, where, after each query $q(t)$, Codemaker must respond with the value of $f(q(t))$. Throughout this paper, we will assume that the series of queries can be made \emph{adaptively}, meaning that the choice of a query $q(t)$ may depend on the answers $f(q(1)), \dots, f(q(t-1))$ to all previous queries. The game is won after $T$ steps if $f$ is uniquely determined by Codemaker's responses to $q(1), \dots, q(T)$. That is, if, for every $f'\in \mathcal{F}\setminus\{f\}$, there exists a $1\leq t \leq T$ such that $f'(q(t))\neq f(q(t))$.

We will use upper case letters $A, B, C, \dots$ to denote query games, and bold upper case letters $\A, \B, \C, \dots$ to denote sets of query games. We will below treat two games as identical if they are \emph{isomorphic} in the sense that they only differ by a relabelling of their queries.

Given two query games $A=(Q_A, \mathcal{F}_A)$ and $B=(Q_B, \mathcal{F}_B)$, we define the \emph{sum game} $A+B$ as the query game $(Q_A\times Q_B, \mathcal{F}_A\times \mathcal{F}_B)$, where we denote by $q_A+q_B$ and $f_A+f_B$ the pairs formed by $(q_A, q_B)\in Q_A \times Q_B$ and $(f_A, f_B)\in \mathcal{F}_A\times \mathcal{F}_B$ respectively, and where we define
$$(f_A + f_B)(q_A + q_B) := f_A(q_A)+f_B(q_B).$$ In other words, $A+B$ is the game in which Codemaker picks a pair of codewords, one from each of $A$ and $B$. For each query in the sum game, Codebreaker asks a query from $A$ and a query from $B$ simultaneously, and is told in return the sum of the responses. For sets of query games $\A$ and $\B$, we define $\A+\B := \{A+B: A\in\A, B\in\B\}.$

We remark that the assumption that query games contain zero-queries ensures that we can query individual terms in a sum game $A+B$ according to $q_A+0$ and $0+q_B$ respectively.

\begin{example} A classical example of a query game is coin-weighing with $n$ coins. We can define this formally as $C_n:=(2^{[n]}, 2^{[n]})$, where we think of a codeword $f\in 2^{[n]}$ as a function $2^{[n]}\rightarrow \mathbb{Z}$ by defining $f(q) := |f\cap q|.$ In other words, Codemaker presents a pile of $n$ coins, and tasks Codebreaker to determine which coins are counterfeit. In each query, Codebreaker is allowed to pick an arbitrary subset of the coins and ask how many in the given set are fake.

It can be seen that this game, up to isomorphisms, satisfies
$$C_n = n\cdot C_1,$$
where $n\cdot C_1$ denotes the $n$-fold sum $C_1+\dots+ C_1$, and, for any $1 \leq k \leq n-1$, we similarly have
$$C_{n} = C_k + C_{n-k}.$$
\end{example}

For a given game $A=(Q_A, \mathcal{F}_A)$ a \emph{strategy} is an adaptively constructed sequence of queries $\left(q(t)\right)_{t=1}^T$. More precisely, for each $1\leq t \leq T$, the choice of query $$q(t)=q\Big(t, f(q(1)), \dots, f(q(t-1))\Big)$$ is allowed to depend on the responses of the preceding queries. We refer to $T$ as the length of the strategy. A strategy is called a \emph{solution} if it can uniquely determine any hidden codeword $f\in\mathcal{F}_A$.

Rather than attempting to construct solutions to query games directly, we often find it more effective to identify a sequence of reductions. To see how this works for query games, consider an instance of a query game $A=(Q_A, \mathcal{F}_A)$, where Codebreaker uses a strategy $\left( q(t)\right)_{t=1}^T$. Suppose that, given the corresponding sequence of answers $f(q(1)), \dots f(q(T))$, Codebreaker concludes that the hidden codeword $f$ is contained in some subset $\mathcal{F}'\subseteq \mathcal{F}_A$. At this point, we can consider $A$ as reduced a query game $B=(Q_B, \mathcal{F}_B)$, provided that each codeword $f'\in \mathcal{F}'$ can be uniquely identified with a codeword $g'\in \mathcal{F}_B$ that is identical to $f'$ up to relabelling and possibly discarding some elements of $Q_A$. Thus any solution of $B$ can be used from this point to solve $A$. We will refer to a strategy as a 
\emph{reduction} from a query game $A$ to a set of query games $\B$, if for any possible answer sequence there is a query game $B\in \B$ with this property. We can state this more formally as follows.

\begin{definition} 
Let $A=(Q_A, \mathcal{F}_A)$ be a query game, and $\B$ a set of query games. We say that $A$ can be reduced to $\B$ in $T$ queries, denoted $A\xrightarrow{T}\B$, if there exists a strategy $\left(q(t)\right)_{t=1}^T$ of length $T$ for $A$, with the following property. For any answer sequence $\mathbf{a}=(a_1, \dots, a_T)$, there exists a choice of a game $B=(Q_B, \mathcal{F}_B)\in\B$, depending only on $\mathbf{a}$, and maps $\varphi=\varphi_\mathbf{a}:Q_B\rightarrow Q_A$ and $\Phi=\Phi_\mathbf{a}:\mathcal{F}_B\rightarrow\mathcal{F}_A$ such that
\begin{enumerate}[(i)]
\item $f'\circ \varphi \in \mathcal{F}_B$, and
\item $\Phi(f'\circ\varphi) = f'$
\end{enumerate}
for any $f'\in\mathcal{F}_A$ consistent with $\mathbf{a}$, that is, for any $f'$ for which the given strategy $\left(q(t)\right)_{t=1}^T$ would produce the answer sequence $\mathbf{a}$.

We denote by $\A\xrightarrow{T} \B$ that $A\xrightarrow{T} \B$ for all $A\in\A$. With slight abuse of notation, we will write $A \xrightarrow{T} B$ and $\A \xrightarrow{T} B$ to denote reduction to the set $\{B\}$.
\end{definition}


When proving that a strategy satisfies the definition to be a reduction, it is often helpful to start by identifying the codewords of game $A$ that are compatible with the answer sequence obtained so far with codewords of a game $B$. Once these codewords are identified, we can then construct maps $\varphi:Q_B\rightarrow Q_A$ and $\Phi:\mathcal{F}_B\rightarrow\mathcal{F}_A$ such that $f'(\varphi(q))=g'$ and $\Phi(g')=f'$ for all identified pairs of codewords $f'$ and $g'$, and for all $q\in Q_B$. One can note that this contrasts the formal definition, where these notions are introduced in the opposite order.

The map $\varphi$ first describes how to convert queries of $B$ to queries of $A$. By property $(i)$, we know that asking any query from $Q_B$ in the game $A$ with hidden codeword $f$ will return answers consistent with asking the same query in $B$ with hidden codeword $g:=f\circ \varphi \in \mathcal{F}_B$. By property $(ii)$, once we determine this codeword $g$ in $B$, the codeword of the original game is given by $\Phi(g)$. Thus, any not yet ruled out codeword $f'$ of $A$ is implicitly identified with the codeword $f'\circ \varphi \in \mathcal{F}_B$, where the map $\Phi$ ensures that no two codewords of $A$ are paired with the same codeword in $B$.

\begin{example}For any query game $A=(Q_A, \mathcal{F}_A)$, any $q\in Q_A$ can be seen as a reduction of length one to the set of games $$A\xrightarrow{1}\{ (Q_A, \{f\in \mathcal{F}_A : f(q)=a\})\},$$
where $a$ goes over all values in $\{f(q) : f\in\mathcal{F}_A\}$, with $\varphi$ and $\Phi$ both being identity maps. Moreover, for any two query games $A, B$ such that $B$ has a solution $\left(q(t)\right)_{t=1}^T$, we have
$$A+B\xrightarrow{T}A,$$
formed by the reduction $0+q(t)$ for $t=1, \dots, T$, $\varphi(q):=q+0$ and $\Phi_\mathbf{a}(f):=f+f_\mathbf{a}$, where $f_\mathbf{a}$ is the unique codeword in $B$ consistent with the answer sequence $\mathbf{a}$.
\end{example}

We will refer to the query game whose only query is the zero-query and whose only codeword is the map $0\mapsto 0$ as the \emph{trivial game} and denote it by $\emptyset$. 
\begin{proposition} Let $T, S\geq 1$ and let $\A, \B$ and $\C$ be sets of query games. The following holds:
\begin{enumerate}[(i)]
\item $\A \xrightarrow{T} \emptyset$ if and only if each $A\in \A$ has a solution of length $T$,
\item if $\A \xrightarrow{T} \B$ and $\B\xrightarrow{S}\C$ then the strategy formed by concatenating the respective reductions forms a reduction $\A \xrightarrow{T+S} \C$.
\end{enumerate}
\end{proposition}
\begin{proof}
$(i)$ Reduction to the trivial game means that, for any $f'\in \mathcal{F}_A$ consistent with a certain answer sequence $\mathbf{a}=(f(q(1), \dots, f(q(T)))$, we have $f'=\Phi_\mathbf{a}(f'\circ \varphi_\mathbf{a})$. But as $f' \circ \varphi_\mathbf{a} \in \mathcal{F}_\emptyset = \{0\mapsto 0\}$, it follows that $f'=\Phi_\mathbf{a}(0\mapsto 0)$ is uniquely determined by $\mathbf{a}$. Hence the hidden codeword is given by $\Phi_\mathbf{a}(0\mapsto 0)$.
 Conversely, any solution can be seen as a reduction to $\emptyset$ by putting $\varphi\equiv 0$ and $\Phi_a\equiv f_\mathbf{a}$ where $f_\mathbf{a}$ denotes the unique solution consistent with answers $\mathbf{a}$.

$(ii)$ Given any $A=(Q_A, \mathcal{F}_A)\in\A$, we start by playing the strategy $\left( q(t) \right)_{t=1}^T$ from $A\xrightarrow{T}\B$ and based on the answer sequence $\mathbf{a}$ compute the corresponding game $B=(Q_B, \mathcal{F}_B)$, and the maps $\varphi:Q_B\rightarrow Q_A$ and $\Phi:\mathcal{F}_B\rightarrow\mathcal{F}_A$. We proceed by making the queries as dictated by the strategy $\left(q'(s)\right)_{s=1}
^S$ from the reduction $B\xrightarrow{S}\C$, according to $\varphi(q'(1)), \varphi(q'(2)), \dots, \varphi(q'(S))$ and based on the answer sequence $\mathbf{b}$ compute the corresponding game $C=(Q_C, \mathcal{F}_C)$, and the maps $\varphi':Q_C\rightarrow Q_B$ and $\Phi':\mathcal{F}_C\rightarrow\mathcal{F}_B.$ Note that for any choice $f\in \mathcal{F}_A$ of the hidden codeword, asking any query $\varphi(q)$ for $q\in Q_B$ is consistent with asking the query $q$ in $B$ with hidden codeword $g:=f\circ \varphi \in \mathcal{F}_B$, hence $\mathbf{b}$, $C$, $\varphi'$ and $\Phi'$ will all be given as if strategy $q'(s)$ was played on $B$ with hidden codeword $g$.

We now argue that this is a well-defined reduction from $A$ to $\C$, as given by the game $C$ and the maps $\varphi \circ \varphi'$ and $\Phi \circ \Phi'$ as given above. Indeed, for any function $f\in \mathcal{F}_A$ consistent with the full answer sequence $(\mathbf{a}, \mathbf{b})$, that is, $f(q(t))=a(t)$ for all $1 \leq t \leq T$ and $f(\varphi(q'(s)))=b(s)$ for all $1 \leq s \leq S$, we have by the first reduction $g:=f\circ \varphi \in \mathcal{F}_B$. By assumption on $f$, we have that $g$ is consistent with answers $\mathbf{b}$ and hence $g\circ \varphi' \in \mathcal{F}_C$. It follows that $\Phi(\Phi'(f \circ \varphi \circ \varphi') )= \Phi(\Phi'(g\circ \varphi')) = \Phi(g) = \Phi(f\circ \varphi)= f,$ as desired.
\end{proof}

With these definitions at hand, we are now ready to present the main technical results of the paper. Given a sequence $\left(\A_n\right)_{n\geq 1}$ of sets of query games, it is natural to ask for a simple condition for the games to allow for short solutions. Our results give conditions for this to hold in terms of finding certain reductions. As we shall see in Section \ref{sec:applications}, such reductions appear naturally in many games.

For the first result, we say that a query game $A$ has a \emph{simple reduction} to $\B$ of length $T$, $A\srightarrow{T}\B$, if there is a reduction $\left(q(t)\right)_{t=1}^T$ such that $f(q(t))\in \{0, 1\}$ for all $f\in \mathcal{F}_A$ and all $1\leq t \leq T$.
\begin{theorem}\label{thm:mainsimple} Let $\left(\A_n\right)_{n\geq 1}$ be a sequence of sets of query games over the integers. If there exists a constant $\alpha>0$ such that
$$\A_1 \srightarrow{\alpha} \emptyset,$$
and, for all $n\geq 2$
$$\A_n \srightarrow{\alpha \cdot n} \A_{\lceil n/2\rceil} + \A_{\lfloor n/2 \rfloor},$$
then $\A_n\xrightarrow{\OO(\alpha \cdot n)} \emptyset.$
\end{theorem}
We will prove this in Section \ref{sec:simpleproof}.

This approach to find solutions to query games can be seen as a variation of the classical divide and conquer method, where Theorem \ref{thm:mainsimple} plays the role of the so-called Master theorem. However, constructing solutions from sequences of reductions as above will require combining reductions in a non-trivial way to form information-dense queries, which does not have an analogue in the usual divide-and-conquer setting. In particular, if one were to do all reductions indicated above to solve $\A_n$ in sequence, this would result in a solution of length $\Theta(\alpha\cdot n \log n)$. Instead, the solutions generated by Theorem \ref{thm:mainsimple} will compress the $\Theta(\alpha \cdot n \log n)$ simple queries into $\Theta(\alpha \cdot n)$ actual queries with the same information content.


Simple reductions appear naturally in some query games, such as coin-weighing and the restriction of Mastermind to permutation codewords. However, in other games, such as Mastermind with general codewords, the requirement that queries must return $0$ or $1$ appears too restrictive. Finding a suitable set of queries may not be significantly easier than finding a solution, or perhaps even such queries do not exist at all. Our second result states that the requirement for the underlying reductions to be simple can be relaxed to allow any queries, but where queries that may return a large value are associated with a higher cost, or weight.

We say that a reduction $\left(q(t)\right)_{t\geq 1}$ from $A$ to $\B$ is \emph{bounded} if there exists a function $b(t)=b(t, f(q(1)), \dots f(q(t-1)))$ such that $0 \leq f(q(t)) \leq b(t)$ for all $t$. We say that this reduction has \emph{weight at most $T$}, $A\brightarrow{T}\B$, if we have $$\sum_{t} \log_2(b(t)+1)^2\leq T,$$
for any choice $f$ of the hidden codeword. We have the following result.
\begin{theorem}\label{thm:mainbounded} Let $\left(\A_n\right)_{n\geq 1}$ be a sequence of sets of query games over the integers. If there exists a constant $\alpha>0$ such that
$$\A_1 \brightarrow{\alpha} \emptyset,$$
and, for all $n\geq 2$
$$ \A_n \brightarrow{\alpha\cdot n} \A_{\lceil n/2\rceil} + \A_{\lfloor n/2 \rfloor},$$
then $\A_n\xrightarrow{\OO(\alpha \cdot n)} \emptyset.$
\end{theorem}
This will be shown in Section \ref{sec:boundedproof}. We remark that, with a more careful analysis, $\log_2( \cdot )^2$ could be replaced by a function of smaller growth. However, as this will not make a difference for our applications, we will not elaborate on this. Furthermore, as any simple reduction is by definition also bounded, Theorem \ref{thm:mainsimple} can be seen as a corollary of Theorem \ref{thm:mainbounded}. We will still prove the two theorems independently. The proof of Theorem \ref{thm:mainsimple} is simpler, and can thus form as a warm-up for the general case. Moreover, the solutions constructed from simple reductions are arguably much simpler, and definitely shorter.

\section{Applications}\label{sec:applications}

In this section, we will show how Theorems \ref{thm:mainsimple} and \ref{thm:mainbounded} can be used to prove the existence of efficient solutions to various classes of query games. We believe that these theorems apply in much greater generality, and we hope that this section can serve as inspiration to find solutions to other query games. In addition to being a powerful tool, our framework also has the benefit that a deep understanding of the underlying coin-weighing schemes is not needed to build solutions.

For any query game $(Q, \mathcal{F})$, we define the \emph{information-theoretic lower bound} as the expression
\begin{equation}\label{eq:qgamelb}
\log_2 |\mathcal{F}| / \log_2 \left( \max_{q\in Q} |\{ f(q) : f\in \mathcal{F}\}|\right).
\end{equation}
A simple counting argument shows that any solution to $(Q, \mathcal{F})$ needs at least this many queries. We will below give multiple examples of query games where our main technical results show that the information-theoretic lower bound is sharp up to constant factors.

\subsection{Coin-weighing variations}\label{ssec:coin}

Recall that $C_n$, coin-weighing with $n$ coins, is the query game $$C_n:=(2^{[n]}, 2^{[n]}),$$ where $f(q)=|f\cap q|$ for all $f, q \in 2^{[n]}$. The information-theoretic lower bound of this query game is given by $n/\log_2(n+1)$. Multiple solutions that solve this problem within a constant factor of this lower bound have been given in \cite{ERcoin,cantor1966determination,lindstrom1965combinatorial}.

Let us reprove this statement using Theorem \ref{thm:mainsimple}.

\begin{proposition}\label{prop:Cn} $C_n \xrightarrow{ \OO(n/\log_2 (n+1))} \emptyset.$\end{proposition}
\begin{proof}
As Theorem \ref{thm:mainsimple} always gives solutions of length $\OO(n)$, we aim to apply it to a sequence $\left(C_{\Theta(n \log_2 n)}\right)_{n\geq 1}$. The key condition to verify is the existence of a simple reduction of length $\OO(n)$ from coin-weighing with $\Theta(n \log_2 n)$ coins to $\Theta( (n/2)\log_2(n/2)+(n/2)\log_2(n/2))$ coins. But as $$(n/2)\log_2(n/2)+(n/2)\log_2(n/2) = n\log_2 n - n,$$
this can be done by weighing $\Theta(n)$ coins one by one, and dividing the remaining ones into two piles of appropriate sizes.

Let us make this more precise. We define $k_n$ recursively by $k_1=1$ and, for all $n\geq 2$, $k_n = n + k_{\lceil n/2 \rceil} + k_{\lfloor n/2\rfloor}$. It can be checked that this implies that $k_n=n\left(1+\lfloor \log_2 n \rfloor\right)$. Thus by the aforementioned strategy,
$$C_{k_n} \srightarrow{n} C_{k_n - n} = C_{k_{\lceil n/2 \rceil}} + C_{k_{\lfloor n/2 \rfloor}},$$
and clearly
$$C_{k_1}=C_1 \srightarrow{1} \emptyset.$$
Hence, by Theorem \ref{thm:mainsimple}, coin-weighing with $k_n=\Theta(n \log n)$ coins can be solved in $\OO(n)$ queries.
\end{proof}

The proof above directly extends to any $n$-fold sum of query games where each game has a simple or a bounded solution.

\begin{proposition}\label{prop:Asum} Let $\A$ be any set of query games such that $\A \brightarrow{\alpha} \emptyset$. Then the $n$-fold sum $n\cdot \A$ can be solved in $\OO(\alpha\cdot n/\log (n+1) )$ queries.
\end{proposition}
\begin{proof}
Let $\B_n:=k_n\cdot \A$, where $k_n$ is as in the proof of Proposition \ref{prop:Cn}. We have that $\B_1=\A\brightarrow{\alpha}\emptyset$, and $$\B_n = k_n\cdot \A \brightarrow{\alpha\cdot n} (k_n-n)\cdot \A = (k_{\lceil n/2 \rceil} + k_{\lfloor n/2 \rfloor})\cdot\A = \B_{\lceil n/2 \rceil} + \B_{\lfloor n/2 \rfloor}.$$
It follows by Theorem \ref{thm:mainbounded} that $\B_n$ has a solution of length $\OO(\alpha \cdot n)$. 
\end{proof}

In fact, we can extend this speed up also to sequences of query games satisfying similar conditions to Theorem \ref{thm:mainbounded}, but where the splitting operation can be performed with sublinear weight. This is analogous to the leaf-heavy regime of the Master theorem for divide-and-conquer recurrences.

\begin{proposition}\label{prop:leafheavy} Let $c>0$ be fixed. Let $\A_n$ be a sequence of sets of query games such that $$\A_n \brightarrow{\alpha\cdot n^{1-c}} \A_{\lceil n/2 \rceil} + \A_{\lfloor n/2 \rfloor}$$ for all $n\geq 2$ and $\A_1 \brightarrow{\alpha}\emptyset$, for some $\alpha>0$. Then $$\A_n \xrightarrow{\OO_c(\alpha \cdot n/ \log (n+1))}\emptyset.$$\end{proposition}
\begin{proof}
We begin by performing the following reductions in sequence
$$ \A_n\brightarrow{} \A_{\lceil n/2\rceil}+\A_{\lfloor n/2\rfloor}\brightarrow{}\dots \brightarrow{} \A_{\Theta(\sqrt{n})} + \dots + \A_{\Theta(\sqrt{n})},$$ 
until what remains is a sum of $2^k=\Theta(\sqrt{n})$ query games, each with parameter $\Theta(n/2^k) = \Theta(\sqrt{n})$, where $k\geq 0$ is an appropriately chosen integer. The total weight of the concatenated reduction is
$$\sum_{i=0}^{k-1} 2^i \cdot \alpha \cdot \Theta\left( (n/2^i)^{1-c} \right) = \Theta_c(\alpha \cdot n^{1-c} \cdot 2^{ck} ) = \Theta_c(\alpha \cdot n^{1-c/2}).$$
In particular, as the weight of a query is at least one, the total reduction has length $\OO_c(\alpha \cdot n^{1-c/2})$.

We now solve the remaining sum of games in parallel. Let $T_k$ denote the minimum weight such that $\A_k\brightarrow{T_k}\emptyset$. We have, by assumption, that $T_1\leq \alpha$ and $T_k \leq \alpha \cdot k^{1-c} + T_{\lceil k/2 \rceil}+T_{\lfloor k/2 \rfloor}$, which implies that $T_k = \OO_c(\alpha \cdot k)$. Applying Proposition \ref{prop:Asum} to the remaining game, with $n'=\Theta(\sqrt{n})$ equal to the number of terms and $\alpha' = \OO_c( \alpha \sqrt{n} )$ equal to the maximum weight needed to solve each individual term, it follows that $\A_n$ can be solved in $\OO_c( \alpha n^{1-c/2} + \alpha n /\log (n+1))=\OO_c(\alpha n/\log (n+1))$ queries, as desired.
\end{proof}

Given the above propositions, we now turn to generalized versions of coin-weighing. First, Dja\v{c}kov \cite{djackov} and Lindstr\"om \cite{Lindstrom75} proposed a sparse version of coin-weighing 
where the total number of counterfeit coins $d$ is fixed beforehand. Formally, we can describe this as the query game
\begin{equation}
C_{n,d} := \left(2^{[n]}, {[n] \choose d}\right),
\end{equation}
for any $1\leq d \leq n$, and where, again, $f(q):=|f\cap q|$ for all $q\in 2^{[n]}$ and $f\in {[n]\choose d}$. 

For $d\leq n/2$, it can be seen that the information-theoretic lower bound is given by $$\log_2 {n \choose d} / \log_2(d+1)=\Omega\left( \log(n/d)\cdot d / \log(d+1) \right),$$ and, by symmetry, for $d\geq n/2$ the bound is given by $\log_2 {n \choose d} / \log_2(n-d+1)$. We remark in the latter case that even though queries can return any value between $0$ and $d$, no one fixed query can span this full range when $d>n/2$.

A second generalization of sparse coin-weighing is to allow the coins to have arbitrary non-negative integer weights where the total weight of the coins is known from the start of the game. Let us denote by $C_n^W$ this version of coin-weighing with $n$ coins where the total weight of coins is $W$. Here, the information-theoretic lower bound is given by 
$$\log_2 {W+n-1 \choose n-1} / \log_2(W+1)=\begin{cases} \Omega\left( \log\left(1+\frac{n}{W}\right) W / \log(W+1)\right) & \text{if }W\leq n,\\\Omega\left( \log\left(1+\frac{W}{n}\right) n / \log(W+1)\right) & \text{if }W\geq n.\end{cases}$$

Coin-weighing with exactly one fake coin is of course a very classical problem, and is canonically solved by binary search. The reader is referred to \cite{bshouty} for an overview of research in the case of $d=2$. In the general case, it was shown in 2000 by Grebinski and Kucherov \cite{GK00} that the information-theoretic lower bounds for $C_{n,d}$ and $C_n^W$ are tight up to constant factors in the full parameter range. As their solution is based on the probabilistic method, it is inherently non-constructive, and it took until 2009 before the first deterministic polynomial time solutions of the same lengths were found by Bshouty~\cite{bshouty}.

We will now use Proposition \ref{prop:leafheavy} to derive optimal solutions to $C_{n,d}$ for all $n, d$ and for $C_n^W$ in the case where $W=\OO(n)$.

\begin{proposition}\label{prop:sparsecoin} Let $n, d, W\geq 1$ be such that $d\leq n$. Then
\begin{enumerate}[(i)]
\item $C_{n,d}$ can be solved in $\OO( \log {n \choose d} / \log(\min(d+1, n-d+1)))$ queries, and
\item $C_n^W$ can be solved in $\OO\left( \max(1, \log(n/W) ) W/\log(W+1)\right)$ queries.\end{enumerate}
\end{proposition}
\begin{proof}
By symmetry between genuine and counterfeit coins, it suffices to show $(i)$ for $d\leq n/2$, in which case the upper bound simplifies to $\OO( d \log(n/d) / \log (d+1))$. This is a special case of $(ii)$ for $W=d$. Hence, the proposition follows if we can prove $(ii)$. To do this, let $\lambda\geq 2$ and define $$\A_m^\lambda:=\{C_{n}^W : W+n/\lambda \leq 2 m\}.$$ We wish to apply Proposition \ref{prop:leafheavy} to the sequence $\left(\A_m^\lambda\right)_{m\geq 1}$.  

We first claim that $\A_1^\lambda \brightarrow{\OO(\log \lambda)}\emptyset$. This is because the only non-trivial games in this set is $C_{n}^1$ for $2\leq n \leq \lambda$, which can be solved in $\lceil \log_2 n\rceil$ simple queries using binary search. Second, let $m\geq 2$ and let $C_n^W$ be any query game in $\A_m^\lambda$. Note that $W\leq 2m$ and $n \leq 2\lambda m$. To build our reduction, first determine the largest integer $0\leq i \leq n$ such that $$ i + f(\{1, 2, \dots, i\}) / \lambda \leq 2\lceil m/2\rceil.$$
As the left-hand side is increasing in $i$, this can be found in $\OO(\log n)$ queries using binary search. Given this $i$, we define $L:=\{1, 2, \dots, i\}$ and $R:=\{i+2,\dots n\}$ and note that the subgames formed by the coins in $L$ and $R$ are in $\A_{\lceil m/2 \rceil}^\lambda$ and $\A_{\lfloor m/2\rfloor}^\lambda$ respectively. Hence, by possibly performing one last query to determine the value of coin $i+1$, we have the desired reduction. As the answer to any query used here is at most $W\leq 2m$, we can bound the weight of this reduction by $\OO(\log (\lambda m) (\log m)^2) = \OO(\log \lambda) \cdot \text{polylog}(m).$

Applying Proposition \ref{prop:leafheavy} to $\A_{m}^\lambda$ with $\alpha=\OO(\log \lambda))$ and, say, $c=\frac12$, it follows that $\A_m^\lambda \xrightarrow{\OO(\log \lambda \cdot m / \log (m+1))}\emptyset.$ For any given $n, W\geq 1$, the result follows by noting that $C_n^W\in \A_W^\lambda$ where $\lambda:=\max\left(2, n/W\right)$.
\end{proof}

\subsection{Mastermind with permutation codeword}

We now turn our attention to Mastermind. In this subsection, as a warm-up, we consider the simplified version of black-peg Mastermind where the number of positions and colors are both $n$, and where the codeword chosen by Codemaker is required to be a permutation of $[n]$.

We will formalize this in the query game setting as
$$M^{\text{perm}}_n = ( (\{0\} \cup [n])^n, S_n),$$
where $S_n$ denotes the set of permutations of $[n]$, and where we interpret each $f \in S_n$ as a function from the queries to the natural numbers according to 
$$f(q):= |\{i : q_i=f_i\}|.$$
We remark that the query game formulation differs slightly from the game mentioned above in that Codebreaker may guess blank at certain positions of queries by putting $q_i=0$.

This problem was previously considered by the author and Su in \cite{MSup22+} where it was shown that it has a solution of length $\OO(n)$. Indeed, as the information-theoretic lower bound is $\log n! / \log( n+1) = \Omega(n)$, this is optimal up to constants. In the same paper, it was shown that black-peg Mastermind with $n=k$ has a randomized reduction to this game (Lemma 3.4, \cite{MSup22+}).

We first prove that allowing blank guesses does not make the problem significantly easier, as blank guesses can be simulated in the usual game by performing a few additional queries.
\begin{claim}\label{claim:szero} If $M_n^{\text{perm}}\xrightarrow{T}\emptyset$, then black-peg Mastermind with a permutation codeword (without blank guesses) can be solved in $T+2\lceil \log_2 n \rceil$ queries.
\end{claim}
\begin{proof}
Observe that if we can identify a query $b \in [n]^n$ that has no matching entries with the hidden codeword, we can simulate guessing blank by replacing each instance of $q_i=0$ by the value $b_i$. So it suffices to show that such a $b$ can be found in $2 \lceil \log_2 n\rceil$ queries.

For each $i=1, \dots, \lceil \log_2 n\rceil$, we make two queries. In the first query, we guess $1$ at all indices where the $i$:th least significant bit in its binary representation is a $0$ and color $2$ where the bit is a $1$. In the second, we reverse the colors. 

We claim that one of these queries must return $0$. As the hidden codeword $f$ is a permutation, there are unique indices $j_1, j_2$ such that $f_{j_1}=1$ and $f_{j_2}=2$. Considering the binary representations of $j_1$ and $j_2$, as they different numbers between $1$ and $n$, there must be some position $i$ in which the numbers have different digits. One of the two corresponding queries will return $0$.
\end{proof}

\begin{proposition} $M_n^\text{perm} \xrightarrow{\OO(n)} \emptyset$.
\end{proposition}
\begin{proof}
We prove this using Theorem \ref{thm:mainsimple}. Note that $M_1^\text{perm}$ is trivial -- the only hidden codeword is ``$1$'', so we only need to consider the reduction of $M_n^\text{perm}$ for $n\geq 2$.

Consider the strategy of length $n$ formed by, for each $1\leq t \leq n$ taking the query
$$q(t) = (t, \dots, t, 0, \dots, 0)$$
which is $t$ in the first $\lceil n/2\rceil$ entries, and $0$ everywhere else, that is, it is $0$ in the last $\lfloor n/2 \rfloor$ entries. Depending on whether color $t$ is in the left or right half of the codeword, the $t$:th query will either return $1$ or $0$, so this strategy is simple. Moreover, given this information, we can treat the remaining game as a sum of two instances of Mastermind whose codewords are each a permutation of the colors $t$ that returned $1$ and $0$ respectively, hence yielding a reduction to $M_{\lceil n/2 \rceil}^\text{perm}+M_{\lfloor n/2 \rfloor}^\text{perm}$.
\end{proof}

\subsection{Black-peg Mastermind}

We now turn to the proof of Theorem \ref{thm:blackpegmm}. Consider full black-peg Mastermind game with $n$ positions and $k$ colors, with no restrictions on the codeword. The formalized version of this game was first considered by Chv\'atal \cite{chvatal1983mastermind}. We will here formalize black-peg Mastermind as the query game
$$M_{n,k} := ( (\{0\} \cup [k])^n, [k]^n),$$
where, same as the previous section, we define
$f(q):=|\{i:f_i=q_i\}|$.
We will write $M_n$ to denote $M_{n,n}$.

We again start by showing that the assumption that Codebreaker can make blank guesses does not make the problem significantly easier by showing that blank guesses can be simulated in the usual version of the game. This uses a slightly more elaborate argument than in the last section.

\begin{claim}\label{claim:bzero}If $M_{n,k} \xrightarrow{T}\emptyset$, then black-peg Mastermind with $n$ slots and $k\geq 2$ colors can be solved in $T+\OO(n/\log n)$ queries.
\end{claim}
\begin{proof}
Analogous to Claim \ref{claim:szero}, it suffices to find a string $b\in [k]^n$ that does not match the codeword in any position. We will find such a $b$ by determining, for each $i\in [n]$, whether $f_i=1, f_i=2,$ or $f_i\in \{3, \dots, k\}.$

We first query the all ones string. Let us denote this by $\mathbf{1}$. After this, for any query $q\in \{1, 2\}^n$ we note that
$$f(q)-f(\mathbf{1})+|\{i\in [n] : q_i=2\}| = \sum_{i=1}^n  \mathbbm{1}_{f_i=q_i}-\mathbbm{1}_{f_i=1}+\mathbbm{1}_{q_i=2},$$
where we note that the summand in the right-hand side is always $0$ for $q_i=1$, and for $q_i=2$ it is $0$ if $f_i=1$, it is $2$ if $f_i=2$ and it is $1$ if $f_i\neq 1, 2$. Hence this problem can be seen as equivalent to coin-weighing with weights $0, 1$ and $2$. This can be solved in $\OO(n/\log n)$ queries using Proposition \ref{prop:Asum}.
\end{proof}

With this statement at hand, it remains to determine the optimal number of queries needed to solve $M_{n,k}$. We have two natural lower bounds for this number. First, the information-theoretic lower bound is $n\log k / \log(n+1)$. Second, for any index $i\in [n]$ it may be that the value of $f_i$ is one of the two last colors to ever get queried in position $i$, meaning that no strategy is guaranteed to determine the codeword in fewer than $k-1$ queries. Putting these together gives us a lower bound of \begin{equation}\label{eq:blackpegmmlb}\Omega\left( \frac{n\log k}{\log n}\right)+k.\end{equation}
We note that the first term dominates when $k=o(n)$, the second term dominates when $k=\omega(n)$, and the terms are comparable when $k=\Theta(n)$.

In order to show that \eqref{eq:blackpegmmlb} is tight, we distinguish two parameter regimes. If $k\leq n^{1-\varepsilon}$ for any small fixed $\epsilon>0$, then it suffices to apply the sparse coin-weighing solution in Proposition \ref{prop:sparsecoin} once for each color. Proposition \ref{prop:mmsparse} shows that this matches the information theoretic lower bound up to constant factors. If $k\geq n^{1-\varepsilon}$, then \eqref{eq:blackpegmmlb} simplifies to $\Omega_\varepsilon (n)+k$. It is not too hard to reduce this problem to showing that $M_n\xrightarrow{\OO(n)}\emptyset$. This is done in Observation \ref{obs:mm} and Proposition \ref{prop:mmn}, which finishes the proof of Theorem \ref{thm:blackpegmm}.

\begin{proposition}\label{prop:mmsparse}
For any $\varepsilon>0$ and $2 \leq k\leq n^{1-\varepsilon}$, the game $M_{n,k}$ can be solved in\\ $\OO_\varepsilon(n \log k / \log n)$ queries.
\end{proposition}
\begin{proof}
For each color $c=1, 2, \dots k$, we first ask the all `$c$' query to determine how many times $d_c$ the color $c$ appears in the hidden codeword. If $d_c\geq 1$, we apply the solution from Proposition \ref{prop:sparsecoin} to determine the positions of the $d_c$ occurrences of the color by making queries in $\{0, c\}^n$. Observe that, since we already know the colors at $d_1+d_2+\dots + d_{c-1}$ positions, we may assume when applying Proposition \ref{prop:sparsecoin} that the total number of coins is $n-\sum_{c'<c} d_{c'}$.

It remains to show that the number of queries,
$$k + \sum_{c=1}^{k} \OO\left( \log {n-\sum_{c'<c} d_{c'} \choose d_c}/\log \min(d_c+1,n-d_c+1)\right)$$
has the prescribed upper bound. For this we consider three cases. First, there could be one color $c$ for which $d_c>n/2$. This will in worst case take as long as the dense coin-weighing problem, that is $\OO(n/\log n)$. Second, the contribution from colors $c$ such that $\sqrt{n/k}\leq d_c \leq n/2$ is at most $$\OO_\varepsilon\left( \sum_c \log {n-\sum_{c'<c} d_{c'} \choose d_c}/\log n\right) = \OO_\varepsilon\left( \log {n \choose d_1, d_2, \dots, d_k} / \log n\right),$$
where the multinomial coefficient can be bounded by $k^n$. Finally, the contribution from colors $c$ where $d_c\leq \sqrt{n/k}$ is at most $\OO(k\cdot \sqrt{n/k} \; \text{polylog}(n))=\OO(n^{1-\varepsilon/2}\text{polylog}(n))$, which is negligibly small.
\end{proof}

It remains to show that $M_{n,k}$ has as solution of length $\OO(n)+k$ for any $k\geq n^{1-\varepsilon}$. In fact, this is true for any $k \geq 2$. We show this in two steps. In the first step, we make queries $(c, c, \dots, c)$ for all $c\in [k]$ to determine the total number of times each color appears in the codeword. In the language of query games, this is a reduction from $M_{n,k}$ to the 
set of query games
$$\MM_n:=\{ M_n(c_1, c_2, \dots) : c_1+c_2+\dots=n\},$$
where $M_n(c_1, c_2, \dots, c_k)$ denotes the query game formed by restricting $M_{n, k}$ to only those codewords which has $c_1$ occurrences of color $1$, $c_2$ of color $2$ and so on. We can think of each query game $M\in\MM_n$ as an instance of Mastermind where Codebreaker is given  \emph{hints} consisting of how many times each color appears in the codeword.

In the second step, we apply Theorem \ref{thm:mainbounded} to solve $
\MM_n$ in $\OO(n)$ in a similar way to $M_n^\text{perm}$. The fact that we apply the framework to $\MM_n$ instead of $M_{n,k}$ means that the reductions need to be set up such that the resulting subgames still contain information about how many times each color appears. In fact, this is necessary to ensure that the reductions of the subgames are still bounded with an appropriate bound on its weight.

\begin{observation}\label{obs:mm} $M_{n,k}\xrightarrow{k} \MM_n$.
\end{observation}
\begin{proof} Using $k$ queries, we determine for each $i\in [k]$ the total number of times $c_i$ each color appears in the codeword. This forms a reduction with reduced game to $M(c_1, \dots, c_k)$.
\end{proof}

\begin{proposition}\label{prop:mmn}$\MM_n \xrightarrow{\OO(n)} \emptyset.$
\end{proposition}
\begin{proof}
We will show this using Theorem \ref{thm:mainbounded}. $\MM_1$ is trivially solved as the hints of any game in this set uniquely determine the codeword. So it suffices to find a bounded reduction of $\MM_n$ for $n\geq 2$.

Fix any $M_n(c_1, c_2, \dots)\in \MM_n$. To form the reduction $M_n(c_1, c_2, \dots)\brightarrow{} \MM_{\lceil n/2\rceil}+\MM_{\lfloor n/2 \rfloor}$, we make the following queries. For each $i \geq 1$ such that $c_i>0$, query $(i, \dots, i, 0, \dots, 0)$, where, similar to the previous subsection, the blocks denote the first $\lceil n/2 \rceil$ and the last $\lfloor n/2 \rfloor$ positions respectively. As we already know the total number of occurrences $c_i$ of color $i$ in the codeword, we can use $c_i$ as an upper bound for this query.

Given this information, we can determine the number of times $c_i'$ each color $i$ appears in the left half of the codeword, and how many times $c''_i = c_i-c'_i$ it appears in the right half, which directly defines the reduced game $M_{\lceil n/2 \rceil}(c'_1, c'_2, \dots) + M_{\lfloor n/2 \rfloor}(c''_1, c''_2, \dots) \in \MM_{\lceil n/2\rceil}+\MM_{\lfloor n/2 \rfloor}$, which shows that this is a reduction. Moreover, as 
$$\sum_{i\,:\,c_i>0} \log_2(c_i+1)^2 \leq \sum_{i\,:\,c_i>0} \OO(c_i) = \OO(n),$$
the reduction is bounded with weight $\OO(n)$. By Theorem \ref{thm:mainbounded}, it follows that  $\MM_n\xrightarrow{\OO(n)}\emptyset$, as desired.
\end{proof}

\subsection{White-peg Mastermind and sparse set query}
We now turn to determining the minimum number of queries needed to solve white-peg Mastermind. Recall that this means that both black- and white-peg information is provided for each query.

Let us first discuss a bit how adding white-peg information changes the setting compared to black-peg Mastermind. The information-theoretic lower bounds for both of these games were discussed in the introduction, where we noted that the lower bound for white-peg Mastermind is a factor $2+o(1)$ smaller than that that of black-peg Mastermind. As we already know that this bound is sharp up to constant factors when $k=\OO(n)$. Thus for any $k=\OO(n)$, white-peg Mastermind can be solved optimally up to constants by simply ignoring the white-peg answers, and apply the black-peg solution derived the in the previous section.

For $k=o(n)$, the picture is even clearer. Observe that any solution to white-peg Mastermind can be turned into a solution to black-peg Mastermind by first making $k$ queries to determine the number of times each color appears in the hidden codeword. As the number of queries needed to solve black-peg Mastermind is $\omega(k)$ when $k=o(n)$ it follows that the length of optimal solutions to black- and white-peg Mastermind only differ by a factor $1+o(1)$ in this range.

Instead, the regime where white pegs shine is when $k=\omega(n)$. Unlike for black-peg Mastermind where Codebreaker may need to try all but one color in each position of the query to get any positive indication for what colors are present in the codeword, in white-peg Mastermind Codebreaker can gain information about about occurrences of a certain color in the codeword by making queries that contain said color in \emph{any} position. This lowers the combinatorial lower bound from $\Omega(k)$ to $\Omega(k/n)$.

Given the results for black-peg Mastermind, it only remains to consider the case where $k = \Omega(n)$, say $k\geq 2n$, in which case we need a solution that matches the lower bound $\Omega(n+k/n)$.

To build our solution, we will determine which colors are actually present in the hidden codeword by observing the total number of (black and white) pegs returned from each query. Since the total number of colors present is clearly at most $n$, we already know how to solve the game optimally from that point. In order to approach this, we will consider a slightly cleaner version of the problem which we call sparse set query, as described in the introduction. For any $k \geq n$, we define this as a query game
$$X_{k,n} := \left( {[k]\choose \leq n}, {[k] \choose \leq n}\right),$$
where $f(q) := |f\cap q|.$ We can think of this as a variation of sparse coin-weighing with $k$ coins, where we know that at most $n$ coins are counterfeit, and where the scale fit at most $n$ coins.

Analogous to Mastermind, there are two natural lower bounds to the number of queries needed to solve this. On the one hand, the information-theoretic bound is $\log \left( \sum_{i=0}^n {k \choose i}\right) / \log (n+1)$ which is $\Omega( n \log \frac{k}{n} / \log (n+1) )$ if $k\geq 2n$ and $\Omega(k/\log(n+1) )$ if $n \leq k< 2n$. On the other hand, we need at least $\lceil k/n \rceil$ queries to weigh each coin at least once.

It is not too hard to see that the game can be solved in $\OO(n/\log(n+1))$ queries when $n \leq k < 2n$. For instance, one may divide $[k]$ into two sets each with at most $n$ elements and apply any efficient solution to the classical coin-weighing problem. So in order to solve this game optimally, it suffices to consider the case when $k \geq 2n$. We will divide this further into two subcases. First, if $k$ is polynomially smaller than $n^2$, say $2n \leq k \leq n^{2-\varepsilon}$ for some $\varepsilon>0$, then a direct application of sparse coin-weighing schemes shows that the information theoretic lower bound is sharp. This will be done in Proposition \ref{prop:nottoosparsesetquery}. Second, if $k \geq n^{2-\varepsilon}$ we can simplify the general lower bound to $\Omega_\varepsilon(n + k/n)$. As we can always pretend to add additional dummy elements to the sets, we may in this latter case always assume that $n \geq k^2$, without affecting the lower bound. We will show in Proposition \ref{prop:balancedsetquery} that $\OO(n + k/n)$ queries suffice in this case. This will conclude the proof of Theorem \ref{thm:sparsesetquery}.

\begin{proposition} \label{prop:nottoosparsesetquery} Let $\varepsilon>0$ be fixed. If $2n \leq k \leq n^{2- \varepsilon}$, then $X_{k,n}\xrightarrow{\OO_\varepsilon (n \log (k/n)/\log n)}\emptyset.$
\end{proposition}
\begin{proof}
We build our solution as follows. Partition $[k]$ into $\lceil k/n \rceil$ sets $D_1, D_2, \dots$ of size at most $n$. For each part $D_i$ we first query the full set to determine the number $d_i$ of counterfeit coins, and then use the sparse coin-weighing scheme from Proposition \ref{prop:sparsecoin} to determine which coins are counterfeit. The number of queries needed for this is
$$\lceil k/n \rceil + \sum_{i=1}^{\lceil k/n \rceil} \OO\left(\log{|D_i| \choose d_i}/\log (d_i+1)\right).$$

In order to bound this expression, let us distinguish terms in this sum depending on the size of $d_i$. If $d_i < n^{\varepsilon/2}$, we can bound the summand by $\OO(d_i \log n)=\OO(n^{\varepsilon/2} \log n)$, and thus the total contribution from such terms is at most $\OO((k/n) \cdot n^{\varepsilon/2} \log n) \leq \OO( n^{1-\varepsilon/2} \log n)$, which is polynomially smaller than $n$. In the case when $d_i \geq n^{\varepsilon/2}$, we have $ \log (d_i+1) = \Theta_\varepsilon(\log n)$, and thus the contribution from these terms are at most
$$\frac{1}{\log n} \OO_\varepsilon \left(\log\left( \prod_{i=1}^{\lceil k/n\rceil} {|D_i| \choose d_i}\right)\right) \leq \frac{1}{\log n} \OO_\varepsilon \left(\log{k \choose n}\right) = \OO_\varepsilon(n \log(k/n) / \log n),$$ as desired.
\end{proof}
\begin{proposition}\label{prop:balancedsetquery}
For any $k\geq n^2,$ we have $X_{k,n}\xrightarrow{\OO(k/n)} X_{n^2,n}$ and for any $n\geq 1$ we have $X_{n^2,n}\xrightarrow{\OO(n)}\emptyset$.
\end{proposition}
\begin{proof}
The first reduction can be done by querying the elements in groups of $n$ and discarding elements in sets that return $0$ until only $n^2$ elements remain.

For the solution to $X_{n^2,n}$, we will apply Theorem \ref{thm:mainbounded}. We start similarly to black-peg Mastermind by making a sequence of queries to obtain upper bounds of parts of the codeword. In this case, fix a balanced $n$-partitioning $[n^2]=D_1\cup \dots \cup D_n$ and query these sets. This forms a reduction $X_{n^2,n} \xrightarrow{n} \mathbf{X}_{n^2,n}$ to the set of games where the number of elements in common between each set $D_i$ and the codeword is fixed. Note that by definition $\mathbf{X}_{1^2,1}$ is trivial, so again it suffices to find a bounded splitting reduction for $n\geq 2$.

Fixing a game $X(D_1, d_1, D_2, d_2, \dots) \in \mathbf{X}_{n^2,n}$ in which the codeword is restricted to satisfying $f(D_i)=d_i$ for all $i\in [n]$, we form this reduction as follows. For each $i\in[n]$ such that $d_i>0$, we partition $D_i$ into three parts $D_i=D_{i,1}\cup D_{i,2} \cup D_{i,3}$ each of size at most $\lfloor n/2 \rfloor$, and query the first two parts to determine $f(D_{i,1}), f(D_{i,2})$, and $f(D_{i,3})=f(D_i)-f(D_{i,1})-f(D_{i,2})$. Given these function values, we greedily build two lists $L$ and $R$ consisting of sets with non-zero overlap with the codeword, under the condition that the total overlap between $f$ and $L$ is at most $\lceil n/2\rceil$ and the total overlap between $f$ and $R$ is at most $\lfloor n/2 \rfloor$, and a set $Z$ consisting of the union of all sets with no overlap with $f$. Note that as any set added to $L$ and $R$ contains at least one element, this implies that $|L|\leq \lceil n/2\rceil$ and $|R|\leq \lfloor n/2 \rfloor$.

This procedure will leave at most one set $D_{i,j}$ with non-trivial overlap with $f$ that does not fit in either $L$ nor $R$ -- if two such sets $D, D'$ existed, then the total overlap of $L \cup D$ and $f$ is greater than $\lceil n/2\rceil$ and the total overlap between $R\cup D'$ and $f$ is greater than $\lfloor n/2\rfloor$, which is a contradiction as the total overlap is at most $n$. In case such a set $D_{i,j}$ occurs, we resolve it by querying its elements one at the time and distribute the singletons between $L$, $R$ and $Z$.

Finally, we use the elements of $Z$ to extend the lists $L$ and $R$ until they contain $\lceil n/2 \rceil$ sets of size $\lceil n/2 \rceil$ and $\lfloor n/2 \rfloor$ sets of size $\lfloor n/2 \rfloor$ respectively. This gives us a reduction to a sum of sparse set query on $\cup_i L_i$ and $\cup_i R_i$ respectively, with $L$ and $R$ describing the hints for the sub-games.

Finally, we need to check that this reduction has weight at most $\OO(n)$. Indeed, using $d_i$ as upper bound when querying $f(D_{i,1})$ and $f( D_{i,2})$ gives a total weight of $2\sum_{i\,:\,d_i>0} \log(d_i+1)^2 = 2\sum_{i\,:\,d_i>0}\OO(d_i) = \OO(n)$. For querying singletons of the remaining set, we make $n$ queries with upper bound $1$, whose weight is clearly also $\OO(n)$, as desired. 
\end{proof}

With an efficient solution to the sparse set game  at hand, we now come back to Theorem \ref{thm:whitepegmm}.

\begin{proof}[Proof of Theorem \ref{thm:whitepegmm}.] As remarked, given Theorem \ref{thm:blackpegmm}, it remains to solve white-peg Mastermind in for $k\geq 2n$ in $\OO(n+k/n)$ queries. We will do so by simulating sparse set query to find the set of colors present in the hidden codeword. Given this, we can apply the black-peg Mastermind algorithm to solve the rest in $\OO(n)$ time.

We begin by identifying a color, call it $z$, that never appears in the hidden codeword. As $k>n$ we know that one such color exists. Indeed, by asking all `$i$' queries  for $i=1, 2, \dots, n+1$, at least one of these must return $0$ black and white pegs.

Given such a $z$, we next aim to determine the set $X\subseteq [k]\setminus\{k\}$ of colors present in the hidden codeword $f\in [k]^n$. Clearly $|X|\leq n$. For any set $Y\subseteq [k]\setminus \{k\}$ of size at most $n$, let $q_Y \in [k]^n$ denote a query such that each color of $Y$ appears once in $q_Y$, and all remaining positions are filled up with the color $z$. The total number of black and white pegs returned for any query $q_Y$ equals $|X\cap Y|$, and thus by the preceding propositions, we can determine $X$ in $\OO( n\log (k/n)/ \log n + k/n)$ queries by simulating sparse set query.

With knowledge of the colors that appear in $f$, we can now apply the black-peg solution from Theorem \ref{thm:blackpegmm} to determine the codeword after an additional $\OO(n)$ queries, as desired.
\end{proof}
\section{Proof of Theorem \ref{thm:mainsimple} }\label{sec:simpleproof}

We now turn to the problem of constructing short solutions for query games satisfying the conditions of Theorem \ref{thm:mainsimple}. As we shall see, after carefully setting up building blocks of the proof, the theorem can be shown by induction on $n$. This is inspired by the recursive construction by Cantor and Mills \cite{cantor1966determination} for coin-weighing.

Before outlining our solution, let us take a moment to reflect on what to aim for. The only operations allowed to us by the theorem is the splitting $$\A_k \srightarrow{\alpha\cdot k} \A_{\lceil k/2 \rceil}+\A_{\lfloor k/2\rfloor},$$  and solving any term $\A_1$ in $\alpha$ simple queries. By applying these operations sequentially, we would form a solution to $\A_n$ in $\OO(\alpha \cdot n \log n)$ simple queries, which is too slow. To speed this up, the aim is that once the game is split into sufficiently many terms we can run the reductions more efficient in parallel. Indeed, as the answer to each individual query is known to be either $0$ or $1$ this parallelism is very similar to classical coin-weighing. A natural approach to improve this is to consider a sequence of reductions layer by layer $$\A_n \rightarrow \A_{n/2}+\A_{n/2} \rightarrow \A_{n/4}+\A_{n/4} + \A_{n/4}+\A_{n/4} \rightarrow \dots, $$
where reduction $i=1, \dots, \OO(\log n)$, can be performed in $\Theta(\alpha \cdot n/2^i \cdot 2^i/i)=\Theta(\alpha \cdot n / i)$ queries by using standard coin-weighing schemes. But as $\sum_{i=1}^{\OO(\log n)} \Theta(\alpha \cdot n /i) = \Theta(\alpha \cdot n \log \log n)$, this is still too slow. The reason why this is still not fast is that the coin-weighing paradigm gives a meaningful speedup at first when the game is already split into many terms. In fact, it only gives the necessary factor $\log n$ speedup once the game has been split into $n^{\Omega(1)}$ terms. While the layer-by-layer approach eventually attains this, it spends too much time initially to split all large (and therefore most costly to split) games at a time when no meaningful parallelism is possible. The way we remedy this is to consider an intermediate game $\PP_n$, the \emph{preprocessed game}, which is formed from $\A_n$ by recursively applying the splitting operation a fraction of terms as they appear in the game. This has the strength that, on the one hand, one can reduce $\A_n$ to $\PP_n$ using $\OO(\alpha \cdot n)$ \emph{simple} queries. On the other hand, $\PP_n$ consists of a sum of polynomially many $\A$-games, meaning we can more directly apply coin-weighing techniques.

Before formally introducing the preprocessed game, we start by stating the lemma that will allow us to run reductions in parallel. Analogous to how we can determine the value of four coins in three weighings, this lemma states that, in the setting of sums of query games we can run four reductions in parallel at the cost (i.e. number of queries) of three of them, provided one reduction is simple. We will be able to construct efficient solutions by applying this recursively. We will choose the exact definition of the preprocessed game below so that it matches this statement nicely.
\begin{lemma} \label{lma:coinspeedup} Let $T>0$ and let $\A, \B, \C, \D, \E$ be sets of query games such that $\A\xrightarrow{T} \emptyset$, $\B\xrightarrow{T} \emptyset$, $\C\xrightarrow{T} \emptyset$, $\D\srightarrow{T} \E$. Then $\A+ \B + \C + \D \xrightarrow{3T} \E$.
\end{lemma}
\begin{proof}
Given games $A, B, C,$ and $D$ from the respective sets, and strategies $\left(q_A(t)\right)_{t=1}^T$, $\left(q_B(t)\right)_{t=1}^T$, $\left(q_C(t)\right)_{t=1}^T$ and $\left(q_D(t)\right)_{t=1}^T$ as in the statement of the lemma, we define a strategy $\left(q(t)\right)_{t=1}^{3T}$ on $A+B+C+D$ by, for each $t=1, \dots, T$, letting
\begin{align*}
q(3t-2) &:= 0+q_B(t)+q_C(t)+q_D(t)\\
q(3t-1) &:= q_A(t)+0+q_C(t)+q_D(t)\\
q(3t) &:= q_A(t)+q_B(t)+0+q_D(t),
\end{align*}
where for each $t$ and each $X\in\{A, B, C, D\}$, the subquery $q_X(t)$ is generated as if past queries have returned $f_X(q_X(1)), \dots, f_X(q_X(t-1))$ where $f=f_A+f_B+f_C+f_D$ denotes the hidden codeword of the sum game.

To see that this is a well-defined strategy, we need to show that $f_X(q_X(t))$ can be uniquely determined by $f(q(1)), \dots, f(q(3t))$. Note that, given the answers to the first $3t$ queries, we can compute
$$f(q(3t))+f(q(3t-1))-f(q(3t-2)) = 2\cdot f_A(q_A(t)) + f_D(q_D(t)),$$
and as $q_D$ is simple, we can determine $f_D(q_D(t))$ by considering the parity of this sum. Given this, we can compute
\begin{align*}
f(q(3t-2)) - f_D(q_D(t)) &= f_B(q_B(t)) + f_C(q_C(t)),\\
f(q(3t-1)) - f_D(q_D(t)) &= f_A(q_A(t)) + f_C(q_C(t)),\\
f(q(3t)) - f_D(q_D(t)) &= f_A(q_A(t)) + f_B(q_B(t)),
\end{align*}
from which $f_A(q_A(t)), f_B(q_B(t))$ and $f_C(q_C(t))$ can be determined by linear algebra.

To show that $\left(q(t)\right)_{t=1}^{3T}$ is indeed the desired reduction, note that as $\left(q_A(t)\right)_{t=1}^T$, $\left(q_B(t)\right)_{t=1}^T$, and $\left(q_C(t)\right)_{t=1}^T$ are solutions for their respective terms, the answers provided uniquely determine $f_A$, $f_B$ and $f_C$ of the codeword. At the same time, for the subgame $D=(Q_D, \mathcal{F}_D)$, the answers to $\left(q_D(t)\right)_{t=1}^T$ determine a game $E=(Q_E, \mathcal{F}_E) \in\E$ and maps $\varphi:Q_E\rightarrow Q_D$ and $\Phi:\mathcal{F}_E\rightarrow \mathcal{F}_D$. One immediately checks that the maps $q \mapsto 0+0+0+\varphi(q)$ for $q\in Q_E$ and $f' \mapsto f_A+f_B+f_C+\Phi(f')$ for $f'\in\mathcal{F}_E$ satisfy the definition of a reduction from $A+B+C+D$ to $\E$, as desired.
\end{proof}

\begin{remark}\label{remark:inductionstrategy}
To illustrate the strength of applying Lemma \ref{lma:coinspeedup}, consider any sets of games $\B_n$ such that $\B_n \srightarrow{3^n} \emptyset$ for all $n\geq 0$. We claim that this implies that
$$\B_n + 3 \cdot \B_{n-1} + 9 \cdot \B_{n-2} + \dots + 3^n \cdot \B_0 \xrightarrow{3^{n+1}} \emptyset.$$
Indeed, this can be shown by induction. For $n=0$ the statement is clearly true with room to spare, and for the induction step one may apply Lemma \ref{lma:coinspeedup} for $\A=\B=\C = \B_{n-1}+3\cdot \B_{n-2}+\dots+3^{n-1}\cdot \B_0$, $\D=\B_n$ and $\E=\emptyset$. Note that this solution resolves a total of $(n+1)3^n$ queries from the subgames by making $3^{n+1}$ queries in the sum game.
\end{remark}

With this lemma at hand, we are now ready to build our solution. For a given sequence of sets of query games $\A_n$ as in Theorem \ref{thm:mainsimple}, we define the \emph{preprocessed} game $\PP_n$, as follows
\begin{equation}\label{eq:preproc}\PP_n:=\begin{cases} \emptyset & \text{if }n\leq 3,\\
\PP_{\lceil \lceil n/2\rceil /2 \rceil} + \PP_{\lfloor \lceil n/2\rceil /2 \rfloor} + \PP_{\lceil \lfloor n/2\rfloor /2 \rceil} + \A_{\lfloor \lfloor n/2\rfloor /2 \rfloor} & \text{if }n\geq 4.\end{cases}\end{equation}
We note that the indices of the four terms in the right-hand side is always equal to $\lceil n/4\rceil$ or $\lfloor n/4 \rfloor$. In particular, $\lceil \lceil n/2 \rceil /2\rceil$ is always equal to the former, and $\lfloor \lfloor n/2 \rfloor /2 \rfloor$ equal to the latter, and the sum of all four indices is $n$.

We now give a simple reduction from $\A_n$ to the preprocessed game. This will be used both as the first step of the solution of $\A_n$, and also as a building block for the second step when solving $\PP_n$.

\begin{claim}\label{claim:simplepreprocess} For any $n\geq 1$, $\A_n \srightarrow{8\alpha n - 7\alpha} \PP_n$.
\end{claim}
\begin{proof}[Proof by induction.] It is straightforward to check that the statement holds for $n\leq 3$. We have
\begin{align*}
\A_1 &\srightarrow{\alpha} \emptyset = \PP_1, \text{ thus }\A_1\srightarrow{\alpha} \PP_1\\
\A_2 &\srightarrow{2\alpha} \A_1 + \A_1 \srightarrow{\alpha + \alpha} \emptyset = \PP_2, \text{ thus }\A_2\srightarrow{4\alpha} \PP_2,\\
\A_3 &\srightarrow{3\alpha} \A_2 + \A_1 \srightarrow{4\alpha + \alpha} \emptyset = \PP_3, \text{ thus }\A_3\srightarrow{8 \alpha} \PP_3.
\end{align*}
Here, the reductions $\A_1+\A_1\srightarrow{} \emptyset$ and $\A_2+\A_1 \srightarrow{} \emptyset$ are formed by solving the subgames one at a time through queries of the form $q+0$ and $0+q'$.

For $n\geq 4$, we have the following sequence of reductions
\begin{align*}
\A_n &\srightarrow{\alpha n} \A_{\lceil n/2 \rceil} + \A_{\lfloor n/2 \rfloor} \\
&\srightarrow{\alpha \lceil n/2 \rceil + \alpha \lfloor n/2 \rfloor} \A_{\lceil \lceil n/2\rceil /2 \rceil} + \A_{\lfloor \lceil n/2\rceil /2 \rfloor} + \A_{\lceil \lfloor n/2\rfloor /2 \rceil} + \A_{\lfloor \lfloor n/2\rfloor /2 \rfloor}\\
&\srightarrow{8 \alpha \lceil 3n/4 \rceil-21\alpha} \PP_{\lceil \lceil n/2\rceil /2 \rceil} + \PP_{\lfloor \lceil n/2\rceil /2 \rfloor} + \PP_{\lceil \lfloor n/2\rfloor /2 \rceil} + \A_{\lfloor \lfloor n/2\rfloor /2 \rfloor}=\PP_n.
\end{align*}
By upper bounding $\lceil 3n/4 \rceil$ by $3n/4 + 3/4$, we see that the concatenation of the above steps forms a simple reduction from $A_n$ to $\PP_n$ in $8\alpha n - 15 \alpha \leq 8\alpha n - 7\alpha$, as desired.
\end{proof}

With this at hand, we are now ready to construct the solution to the preprocessed game.

\begin{claim}\label{claim:simplepostprocess} For any $n\geq 1$, $\PP_n \xrightarrow{8\alpha(n-1)} \emptyset$. \end{claim}
\begin{proof}[Proof by induction.] As $\PP_1=\PP_2=\PP_3=\emptyset$, the statement is trivial for $n\leq 3$. For $n\geq 4$, define $$T=\max\left(8\alpha\lceil n/4 \rceil-8\alpha, 8\alpha\lfloor n/4 \rfloor-7\alpha\right),$$ and note that, we have the reductions
$$\PP_n = \underbrace{\PP_{\lceil \lceil n/2\rceil /2 \rceil}}_{\xrightarrow{T} \emptyset} + \underbrace{\PP_{\lfloor \lceil n/2\rceil /2 \rfloor}}_{\xrightarrow{T} \emptyset} + \underbrace{\PP_{\lceil \lfloor n/2\rfloor /2 \rceil}}_{\xrightarrow{T} \emptyset} + \underbrace{\A_{\lfloor \lfloor n/2 \rfloor /2 \rfloor}}_{\srightarrow{T}\PP_{\lfloor \lfloor n/2 \rfloor /2 \rfloor}},$$
where the solutions of the first three terms follow from the induction hypothesis, and the reduction of the last term from Claim \ref{claim:simplepreprocess}.  Hence by Lemma \ref{lma:coinspeedup}, we get
$$\PP_n \xrightarrow{3T} \PP_{\lfloor n/4\rfloor} \xrightarrow{8\alpha(\lfloor n/4 \rfloor-1) } \emptyset,$$
where the last step is, again, by the induction hypothesis.

Hence, it suffices to check that the length $3T+8\alpha(\lfloor n/4\rfloor -1)$ of the concatenated strategy is always at most $8\alpha(n-1)$. Indeed
$$ 24\alpha\lceil n/4 \rceil - 24\alpha + 8\alpha (\lfloor n/4\rfloor -1) \leq 24\alpha (n+3)/4  - 24\alpha + 8\alpha ( n/4 -1) = 8\alpha n - 14 \alpha,$$
$$ 24\alpha\lfloor n/4 \rfloor - 21\alpha + 8\alpha (\lfloor n/4\rfloor -1) \leq 32 \alpha \cdot n / 4 -29 \alpha = 8\alpha n - 29 \alpha.$$

\end{proof}

To conclude the proof of Theorem \ref{thm:mainsimple} we note by Claims \ref{claim:simplepreprocess} and \ref{claim:simplepostprocess} that
$$ \A_n \srightarrow{8\alpha\cdot n-7\alpha} \PP_n \xrightarrow{8\alpha(n-1)} \emptyset,$$
which indeed combines to a solution of length $16\alpha \cdot n - 15\alpha$. \qed

\section{Proof of Theorem \ref{thm:mainbounded}}\label{sec:boundedproof}

In this section, we modify the arguments from the last section to allow for queries whose answer can be larger than one, provided sensible upper bounds on each query are known beforehand. We will construct the solution in a very similar way to the previous section. We again use the preprocessed game \eqref{eq:preproc} and define our solution around it recursively as before. The main challenge is how to generalize Lemma \ref{lma:coinspeedup}.


We will do this by showing that it is possible to replace the simple reduction in Lemma \ref{lma:coinspeedup} with a reduction $\left(q(t)\right)_{t=1}^T$ with the property that $f(q(t)) \in [0, r(t)]$ for all $1\leq t\leq T$ and for a carefully chosen (non-adaptive) function $r:\{1, 2, \dots\} \rightarrow \{0, 1, 2, \dots\}$. This comes with a few caveats. First, we can no longer allow arbitrary solutions for the first three terms, $\A, \B, \C$, but the solutions need to have additional structure. Second, the choice of the function $r(t)$ will depend on the solutions to $\A, \B$ and $\C$, so, rather, we will need to consider a set $\mathcal{R}$ and require that $r$-bounded reductions exist for all functions $r(t)$ in $\mathcal{R}$.

The set $\mathcal{R}$ we will work with is the set of functions $r(t)$ that satisfy that there exists a sequence $a_1, a_2, \dots$ of integers such that
$$r(t)=\begin{cases} 2^{2^i}-1 &\text{ if }t\equiv a_i \text{ mod }4^i\text{ for some }i\geq 1\\ 0 &\text{ otherwise}.\end{cases}$$
We say that $A$ has an $\mathcal{R}$-reduction to $\B$ in $T$ steps, $A\xrightarrow[\mathcal{R}]{T}\B$, if there exist $r(t)$-bounded reductions from $A$ to $\B$ in $T$ steps for all $r(t)\in\mathcal{R}$.

\begin{proposition}\label{prop:Rbounded}
If $A\brightarrow{T} \B$, then $A \xrightarrow[\mathcal{R}]{4T}\B$.
\end{proposition}
\begin{proof}
Given any bounded reduction $\left(q(t)\right)_{t\geq 1}$ with an adaptive upper bound $\left(b(t)\right)_{t\geq 1}$ where $\sum_t \log_2(b(t)+1)^2 \leq T$ and any $r(t)\in \mathcal{R}$ we construct an $r(t)$-bounded reduction $\left(q'(t)\right)_{t=1}^{4T}$ by appropriately padding with zeros: Initially let $s=1$. For each $t=1, 2, \dots$ until the bounded reduction terminates,  if $b(s)\leq r(t)$ we let $q'(t):=q(s)$ and increase $s$ by one. Otherwise, we let $q'(t):=0$ and do not update $s$.

As the new strategy does not change which queries are being made, it is still a reduction, and it is by construction $r(t)$-bounded. It suffices to show that this procedure always stops with $t\leq 4T$. Suppose at a time $t\geq 1$, the strategy has just made the query $q(s-1)$, and we know that $b(s) \leq 2^{2^i}-1$ for some suitable $i\geq 1$, then by definition of $\mathcal{R}$, the strategy will make the query $q(s)$ at most $4^i$ time steps later. For any $b(s)\geq 1$, we can always find such an $i\geq 1$ with $2^i \leq 2 \log_2(b(s)+1)$, or, equivalently, $4^i \leq 4 \log_2(b(s)+1)^2$. Hence, the total amount of time needed to make all queries is $4 \sum_s \log_2(b(s)+1)^2 \leq 4T$, as desired.
\end{proof}

The rest of the section will be dedicated to proving the following statement. Theorem \ref{thm:mainbounded} follows immediately by taking $\beta:=4\alpha$ and plugging in Proposition \ref{prop:Rbounded}.
\begin{theorem}\label{thm:stackable} Let $\mathcal{R}$ be as above, and let $\A_n$ be a sequence of sets of query games such that
$$\A_n \xrightarrow[\mathcal{R}]{\beta \cdot n} \A_{\lceil n/2 \rceil}+\A_{\lfloor n/2 \rfloor},$$
and $\A_1 \xrightarrow[\mathcal{R}]{\beta}\emptyset$ for some constant $\beta\geq 1$. Then $\A_n\xrightarrow{\OO(\beta \cdot n)}\emptyset.$
\end{theorem}

In order to show this, we need yet another special kind of reduction. For any function $p(t):\{1, 2, \dots\}\rightarrow\{1, 2, \dots\}$ we say that a game $A$ has a \emph{$p(t)$-predictable} reduction to $\B$ in $T$ queries, $A\xRightarrow[p]{T}\B$, if there exists a reduction such that the congruence class of $f(q(t))$ modulo $p(t)$ is uniquely determined by $f(q(1)), \dots, f(q(t-1))$ for all $1\leq t \leq T$.

In order to relate $p(t)$-predictable reductions to $r(t)$-bounded reductions, we need to introduce some additional structure. Let $\mathcal{Q}$ be the set of pairs of functions $(x(t), y(t))_{t\geq 1}$ such that
\begin{itemize}
\item $x(t)$ is a power of two and $0 \leq y(t) \leq x(t)-1$ for all $t\geq 1$, and
\item for any pair of integers $i\geq 1$ and $0\leq j \leq 2^i-1$, there exists a constant $a_{ij}$ such that
$$(x(t), y(t)) = (2^i, j) \Leftrightarrow t \equiv a_{ij}\text{ mod }4^i.$$
\end{itemize}
It is not too hard to see that such pairs of functions exist. For instance, one can imagine going through the pairs $(i, j)$ as above in lexicographical order and letting $a_{ij}$ be the smallest integer $a\geq 1$ such that $a \not\equiv a_{i'j'}$ mod $4^{i'}$ for any $(i', j') < (i, j)$. Such an $a$ always exists as the total density of integers covered by congruence classes $(i', j')<(i, j)$ is $\sum_{(i',j')<(i,j)} 4^{-i'} < 1$.

For any pair $(x(t), y(t))_{t\geq 1}\in \mathcal{Q}$ say that a reduction is $(x(t), y(t))$-bounded if it is $r(t)$-bounded for the function $r(t):= \mathbbm{1}_{y(t)=0} (2^{x(t)}-1)$. We say that a reduction is $(x(t), y(t))$-predictable if it is $p(t)$-predictable for $p(t):= 2^{y(t)}$. We will write $A \xrightarrow[\mathcal{Q}]{T} \B$ and $A \xRightarrow[\mathcal{Q}]{T} \B$ to denote that $(x(t), y(t))$-bounded reductions and $(x(t), y(t))$-predictable reductions respectively exists for all $(x(t), y(t))\in \mathcal{Q}$. Note that if $A\xrightarrow[\mathcal{R}]{T} \B$ then we also have $A\xrightarrow[\mathcal{Q}]{T} \B$.

\begin{lemma}\label{lma:Rcoinspeedup} If $\A\xRightarrow[\mathcal{Q}]{T} \emptyset$, $\B\xRightarrow[\mathcal{Q}]{T} \emptyset$, $\C\xRightarrow[\mathcal{Q}]{T} \emptyset$, and $\D \xrightarrow[\mathcal{Q}]{T} \E$, then $\A+\B+\C+\D \xRightarrow[\mathcal{Q}]{3T+6} \E$.
\end{lemma}
\begin{proof}
Let $A, B, C, D$ be query games from the respective sets, and let $(x(t), y(t))\in \mathcal{Q}$. We construct a $(x(t), y(t))$-predictable reduction from $A+B+C+D$ to $\E$ as follows.

Let $q_A(t)$, $q_B(t)$ and $q_C(t)$ denote three $(x_A(t), y_A(t))$-predictable, $(x_B(t), y_B(t))$-predictable and $(x_C(t), y_C(t))$-predictable reductions respectively where
\begin{equation}\label{eq:interlace}\begin{split}
(x_A(t), y_A(t)) &= (x(3t), y(3t)-1\text{ mod }x(3t)),\\
(x_B(t), y_B(t)) &= (x(3t+1), y(3t+1)-1\text{ mod }x(3t+1)),\\
(x_C(t), y_C(t)) &= (x(3t+2), y(3t+2)-1\text{ mod }x(3t+2)).
\end{split}\end{equation}
Here we use $y-1\text{ mod }x$ denotes $y-1$ if $y>0$ and $x-1$ if $y=0$.  It can be checked from the definition of $\mathcal{Q}$ that these sequences are contained in $\mathcal{Q}$. Let $q_D(t)$ be a $(x_{t},y_{t})$-bounded reduction from $D$ to $\E$.

We construct a reduction of length by $3(T+2)$ from $A+B+C+D$ to $\E$, by, for each $1\leq t \leq T+2$ letting
\begin{align*}
q(3t-2)&:=0+q_B(t-1)+q_C(t-2)+q_D(t)\\
q(3t-1)&:=q_A(t)+0+q_C(t-2)+q_D(t)\\
q(3t)&:=q_A(t)+q_B(t-1)+0+q_D(t),
\end{align*}
where we put $q_X(t):=0$ for $X\in\{A, B, C, D\}$ whenever $t\leq 0$ or $t>T$.

We first show that this is a well-defined reduction to $\E$. Let $f=f_A+f_B+f_C+f_d$ denote the hidden codeword. Similar to Lemma \ref{lma:coinspeedup}, it suffices to prove that $f(q(1)), \dots, f(q(3t))$ uniquely determines the subqueries $$f_A(q_A(t)),\; f_B(q_B(t-1)),\; f_C(q_C(t-2)),\; f_D(q_D(t))$$ for each $1\leq t\leq T+2$. We will show this by induction on $t$. For any given $1 \leq t \leq T+2$, we may assume that we have already determined $f_A(q_A(s))$ for all $s<t$, $f_B(q_B(s))$ for all $s<t-1$, and $f_C(q_C(s))$ for all $s<t-2$. Observe that it is sufficient to determine $f_D(q_D(t))$, as the answers to the remaining subqueries then can be extracted by linear algebra. In particular, this trivially holds for any $t$ such that $y(t)>0$ as then $f_D(q_D(t))=0$.

We will resolve the cases where $y(t)$ by considering the congruence of $t$ modulo $3$. If $t=3s$, then as $q_A(t)$ is $(x_A(t), y_A(t))$-predictable, we can determine uniquely the congruence class of $f_A(q_A(t))=f_A(q_A(3s))$ modulo $2^{y_A(s)}=2^{x(3s)-1}=2^{x(t)-1}$ from past subqueries. In particular, as
$$-f(q(3t-2)+f(q(3t-1)+f(q(3t)) = 2\cdot f_A(q_A(t)) + f_D(q_D(t)),$$
this allows us to determine uniquely $f_D(q_D(t))$ mod $2^{x(t)}$, and as $0 \leq f_D(q_D(t)) \leq 2^{x(t)}-1$, this uniquely determines $f_D(q_D(t))$.

Similarly, the cases $t=3s+1$ and $t=3s+2$ can be resolved by using the predictability of $q_B(t)$ and $q_C(t)$ together with
$$f(q(3t-2)-f(q(3t-1)+f(q(3t)) = 2\cdot f_B(q_B(t-1)) + q_D(t),$$
$$ f(q(3t-2)+f(q(3t-1)-f(q(3t)) = 2\cdot f_C(q_C(t-2)) + q_D(t).$$
We conclude that $q(t)$ is a well-defined reduction to $\E$.

Finally, we need to check that $q(t)$ is $(x(t), y(t))$-predictable, that is, $f(q(3t))$ mod $2^{y(t)}$ is uniquely determined from $f(q(1)), \dots, f(q(3t-1))$. Note that this is a trivial statement if $y(t)=0$. For any $t\geq 1$ such that $y(t)>0$ we know that $f_D(q_D(t))=0$ and again by case distinction on $t$ mod $3$, one can see that $f(q(1)), \dots f(q(3t-1))$ uniquely determine the congruence class of $f_A(q_A(t)), f_B(q_B(t-1))$ or $f_C(q_C(t-2))$ mod $2^{y(t)-1}$. The statement follows as before by linear algebra.
\end{proof}

\begin{proposition}\label{prop:Rconcat} Let $\mathcal{Q}$ be as above, let $T, S\geq 1$ and let $\A, \B$ and $\C$ be sets of query games. The following holds:
\begin{enumerate}[(i)]
\item If $\A\xrightarrow[\mathcal{Q}]{T} \B$ and $\B\xrightarrow[\mathcal{Q}]{S}\C$, then $\A\xrightarrow[\mathcal{Q}]{T+S}\C$,
\item if $\A\xRightarrow[\mathcal{Q}]{T}\B$ and $\B\xRightarrow[\mathcal{Q}]{S}\C$ and $T$ is divisible by three, then $A\xRightarrow[\mathcal{Q}]{T+S}\C.$
\end{enumerate}
\end{proposition}
\begin{proof} Let $(x(t), y(t))\in \mathcal{Q}$ be given. Observe that any linear shift $((x(t+a), y(t+a))$ is also contained in $\mathcal{Q}$. For $(i)$ we concatenate a $(x(t),y(t))$-bounded reduction from $\A$ to $\B$ with a $((x(t+T), y(t+T))$-bounded reduction from $\B$ to $\C$. For $(i)$ we concatenate a $((x(t), y(t))$-predictable strategy from $\A$ to $\B$ with a $(x(T/3+t), y(T/3+t))$-predictable strategy from $\B$ to $\C$.
\end{proof}

We now turn to proving Theorem \ref{thm:stackable}. Let $\A_n$ be any sequence of sets of query games such that
$$\A_n\xrightarrow[\mathcal{Q}]{\beta\cdot n} \A_{\lceil n/2 \rceil}+\A_{\lfloor n/2 \rfloor},$$
and $\A_1 \xrightarrow[\mathcal{Q}]{\beta} \emptyset$ for some $\beta\geq 1$. Similarly, as before, we define the preprocessed game $\PP_n$ according to
\begin{equation*}\label{eq:preproc2}\PP_n:=\begin{cases} \emptyset & \text{if }n\leq 3,\\
\PP_{\lceil \lceil n/2\rceil /2 \rceil} + \PP_{\lfloor \lceil n/2\rceil /2 \rfloor} + \PP_{\lceil \lfloor n/2\rfloor /2 \rceil} + \A_{\lfloor \lfloor n/2\rfloor /2 \rfloor} & \text{if }n\geq 4.\end{cases}\end{equation*}

The following statement is proven identically to Claim \ref{claim:simplepreprocess}, using Proposition \ref{prop:Rconcat} to concatenate the strategies.

\begin{claim}\label{claim:Rpreprocess} For any $n\geq 1$, $\A_n \xrightarrow[\mathcal{Q}]{8\beta n-7\beta}\PP_n.$\qed
\end{claim}

Secondly, we can show that $\PP_n$ has an $\mathcal{Q}$-predictable solution in $8\beta(n-1)$ time in the same way as Claim \ref{claim:simplepostprocess} but replacing Lemma \ref{lma:coinspeedup} by Lemma \ref{lma:Rcoinspeedup}. Note that the latter lemma adds $6$ additional queries to the length of the combined reduction, but the inequalities at the end of Claim \ref{claim:simplepostprocess} has room to spare for this.

\begin{claim} For any $n\geq 1,$ $\PP_n \xRightarrow[\mathcal{Q}]{8\beta(n-1)}\emptyset.$ \qed.
\end{claim}

By concatenating the reductions from the two claims above, we conclude that $A_n$ has a solution of length $8\beta n - 7 \beta + 8\beta(n-1) = 16 \beta n - 15 \beta$, which concludes the proof of Theorem \ref{thm:stackable}. \qed

\section{Conclusion}\label{sec:conclusion}


In conclusion, we have demonstrated the versatility of the framework presented in Section \ref{sec:querygame} for finding optimal solutions to a wide range of guessing games over integers. Hopefully, this paper can act as a guide to finding solutions to further games in the future.

While this paper has focused on proving the existence of solutions, with little discussion on how these may look, it is worth noting that any solution given by Theorems \ref{thm:mainsimple} and \ref{thm:mainbounded} can be efficiently implemented, provided that the elementary reductions have efficient implementations. More precisely, in order to show that a solution given by either of these theorems has a polynomial time implementation, it is sufficient to show that there is an efficient way to implement the reductions $\left(q(t)\right)_{t=1}^T$, the pairs of query games $A+B$ produced by the reductions, and the maps $\varphi, \Phi$. In fact, the map $\Phi$ is only needed to ensure that decoding the codeword can be done efficiently.

Additionally, while we have not optimized constants in our solutions, it would be interesting to see how close the lower bounds one can push these constructions. For example, for black-peg Mastermind with $n$ colors and positions, the argument in the paper shows that the minimum number of queries needed lies between $n$ and $129n+o(n)$. It seems likely that the length of the optimal solution for this problem grows as $c\cdot n$ for some small constant $c>1$. Resolving this completely may be difficult due to the long-standing remaining factor $2$ gap for adaptive coin-weighing, but it would be interesting to see how much this gap could be narrowed.

In our main technical results, Theorems \ref{thm:mainsimple} and \ref{thm:mainbounded}, we require the splitting operation to divide $\A_n$ exactly in half (up to rounding). It is natural to ask whether the same conclusion holds if we can only ensure an inexact split $$\A_n \brightarrow{\alpha \cdot n} \bigcup_{k=\lceil \varepsilon n\rceil}^{\lfloor(1-\varepsilon)n\rfloor} \left( \A_k + \A_{n-k} \right),$$
for some $0 < \varepsilon \leq \frac12$. Indeed, to show this, we can define $$\B_n:=\bigcup_{n_1+n_2+\dots = n} \left( \A_{n_1}+\A_{n_2}+\dots\right),$$
where $(n_1, n_2, \dots)$ go over all integer compositions of $n$. By iteratively applying the inexact splitting for $\A_n$ we obtain a bounded reduction $\B_n \brightarrow{\OO(\varepsilon^{-1} \alpha \cdot n)} \B_{\lceil n/2 \rceil}+\B_{\lfloor n/2 \rfloor}$, which implies $\B_n\xrightarrow{\OO(\varepsilon^{-1} \alpha \cdot n) } \emptyset$, and thus $\A_n\xrightarrow{\OO(\varepsilon^{-1} \alpha \cdot n)}\emptyset$.

Another way the main technical results can be strengthened is to consider what happens if the solution to $\A_1$ has a high weight, say $\A_1\brightarrow{\beta} \emptyset$ for some large $\beta>0$. In this case, we can still use the solution provided by Theorem \ref{thm:mainbounded} of length $\OO(\alpha\cdot n)$ to reduce $\A_n$ to $n\cdot \A_n$. Then using Proposition \ref{prop:Asum}, we can solve this sum using $\OO(\beta \cdot n / \log n)$ additional queries. In other words, in order to show that $\A_n\xrightarrow{\OO(n)}\emptyset$, the bound on the reduction $\A_1\brightarrow{}\emptyset$ can be relaxed from $\OO(1)$ to $\OO(\log n)$.

Finally, we note that the query game framework presented in this article is reminiscent of a communication theory problem introduced by Chang and Weldon \cite{CW79}, in which multiple independent users want to communicate by sending bits over a channel that adds up their answers. The schemes derived in this paper may be of independent interest within this line of investigation. In particular, the solution derived in Remark \ref{remark:inductionstrategy} can be directly reinterpreted as a coding scheme in this setting where the $i$:th user transfers information at rate $\Theta(1/i)$.

\section*{Acknowledgements}
I would like to thank the anonymous referees for their careful reading of the manuscript and many thoughtful comments and suggestions.

\bibliographystyle{abbrv}
\bibliography{querygames}

\end{document}